\documentclass[reqno]{amsart}
\usepackage[bookmarks=false]{hyperref} 
\usepackage{amssymb}
\usepackage{amsmath}
\usepackage{amsthm} 
\usepackage{mathtools}
\usepackage{color} 
\usepackage[svgnames]{xcolor}
\usepackage[utf8]{inputenc}
\usepackage[T1]{fontenc}
\usepackage{graphicx}
\usepackage{placeins}
\usepackage{vmargin}
\usepackage{setspace}

\hypersetup{
  colorlinks=true,
  linkcolor=Blue  ,          % color of internal links (change box color with linkbordercolor)
  citecolor=Red,        % color of links to bibliography
  filecolor=Magenta,      % color of file links
  urlcolor=Green,           % color of external links
pdfpagemode=UseOutlines,
pdftitle={},
pdfauthor={Tin Van Phan <phanvantin1@gmail.com>},
pdfsubject={Derivative nonlinear Schr\"odinger equation},
pdfkeywords={derivative nonlinear Schr\"oginer equation, stability, solitons, periodic waves} 
}

\numberwithin{equation}{section}
\numberwithin{figure}{section}

\newtheorem{theorem}{Theorem}[section]

\newtheorem{lemma}[theorem]{Lemma}

\newtheorem{corollary}[theorem]{Corollary}

\newtheorem*{assumption a}{Assumption A}
\newtheorem*{assumption b}{Assumption B}
\newtheorem*{assumption c}{Assumption C}

\theoremstyle{definition}

\theoremstyle{remark}
\newtheorem{remark}[theorem]{Remark}

\DeclarePairedDelimiter{\norm}{\lVert}{\rVert}

\newcommand{\N}{\mathbb{N}}
\newcommand{\R}{\mathbb{R}}

\renewcommand{\leq}{\leqslant}
\renewcommand{\geq}{\geqslant}

\DeclareMathAlphabet{\mathpzc}{OT1}{pzc}{m}{it}

\renewcommand{\Im}{\mathcal I\!\mathpzc{m}}

\begin{document}

\title[]{Orbital Stability of Solitons and Scattering Theory for the Perturbed Derivative Nonlinear Schrödinger Equation}

\author[Phan Van Tin]{Phan Van Tin}

\address{VNU University of Engineering and Technology, Vietnam National University
  \newline\indent
  144 Xuan Thuy, Cau Giay, Ha Noi
  }
\email{\url{phanvantin1@gmail.com}}


\date{\today}

\begin{abstract} 
We consider the following derivative nonlinear Schr\"odinger equation with a single power-type perturbation
\begin{equation*}
i\partial_tu+\partial_x^2u+i|u|^2\partial_xu+b |u|^pu=0,
\end{equation*}
with $b\geq 0$ and $p\geq 4$. When $b=0$ or $p=4$, the equation possesses a family of two-parameter solitons; see, for instance, \cite{CoOh06,Oh14}. Moreover, the authors established the orbital stability/instability of these solitons. In \cite[Corollary]{CoOh06}, a criterion for orbital of solitons was proved. Using the explicit formula of solitons to compute the necessary quantities, the authors verify that the solitons are orbitally stable across most of their range of existence (i.e., $c^2 < 4\omega$). When $b\neq 0$ or $p>4$, an explicit formula for the soliton profile is unavailable, making it difficult to verify the criterion in \cite{CoOh06}. In this paper, we prove the soliton profile varies smoothly with respect to parameters $b$ and $p$. More precisely, we show that the solitons change slowly when $b$ is sufficiently small or $p$ is sufficiently close to $4$ (Lemma \ref{lm2.5}, Lemma \ref{lm2.6}). Consequently, we obtain the orbital stability of solitons in these cases (Theorem \ref{thm:main_1}, Theorem \ref{thm:main_2}). In the borderline case ($c=2\sqrt{\omega}$), the soliton also varies smoothly with respect to the parameter $b$ (Lemma \ref{lm:behave_solitons_zero mass}). However, the orbital stability/instability of solitons in this case remains an open question.

The solutions to \eqref{eq1} are fail to scatter even for small initial data (Lemma \ref{lm:scattering_theory_1}), where the original result for $b=0$ was proved in \cite{BaWuXu20}. In \cite{HaOz94}, the authors were successfully constructed the modified wave operator when $b=0$. Moreover, the authors proved that under a smallness of solution, if the solution scatters then it is identically zero. In our case, the presence of perturbation $b|u|^pu$ does not alter the results in \cite{HaOz94} (Theorem \ref{thm:modified scattering operator}, Theorem \ref{thm:non exists non trivial scattering solutions}). 
\end{abstract}

\maketitle
\tableofcontents

\section{Introduction}

We consider the following derivative nonlinear Schr\"odinger equation with a single power perturbation
\begin{equation}\label{eq1}
i\partial_tu+\partial_x^2u+i|u|^2\partial_xu+b |u|^pu=0,
\end{equation}
with $p\geq 4$. The unknown function $u$ is a complex-valued function of time $t\in\R$ and space $x\in\R$. The equation \eqref{eq1} admits three conservation quantities with respect to time:
\begin{align}
E(u)&=\frac{1}{2}\int_{\R}|u_x|^2 dx+\frac{1}{4}\Im\int_{\R}|u|^2\overline{u}u_x dx-\frac{b}{p+2}\int_{\R}|u|^{p+2}dx, \quad \text{ (Energy) }\\
M(u)&=\frac{1}{2}\int_{\R}|u|^2dx,\quad \text{ (Mass) }\\
P(u)&=-\frac{1}{2}\int_{\R}\overline{u}u_x dx. \quad \text{ (Momentum) }
\end{align}
When $b=0$, the equation \eqref{eq1} is called the derivative nonlinear Schödinger equation that was introduced in Plasma Physics as a simplified model for Alfven waves propagation, see \cite{Mjølhus_76,SuSu99}. Since then, it has attached a lot of attention from the mathematical community.\\

The local well posedness of \eqref{eq1} in the energy space $H^1$ was obtained in \cite{HaOZ92,Ha93}. In low regularity spaces $H^s(\R)$ ($s\geq\frac{1}{2}$), local well posedness was studied in \cite{Takaoka01}. On the torus, local existence in $H^s(\mathbb{T})$ was proved in \cite{Her06}. Moreover, by the use of a conservation law, the problem was showed to be globally well-posed for $s\geq 1$ and for data which is small in $L^2(\R)$. Results for $s<\frac{1}{2}$ were obtained by \cite{Takaoka16}. Recently, the authors \cite{HaOzVi25} proved that solutions were global in $H^2(\R)$ under a smallness assumption of initial data. In \cite{Wu13,Wu15}, the author proved global existence in $H^1(\R)$ for initial data $u_0$ having $L^2$-norm smaller than $\sqrt{4\pi}$. The method introduced in \cite{Wu13,Wu15} was used in the torus in \cite{MoOh15}. Recently, using investigating integrability techniques, global existence was proved in weight Sobolev spaces see \cite{LiSiSu13,LiPeSu18,PeSh18}. It then was shown in \cite{BaPe22} that solutions are global in $H^{\frac{1}{2}}$ for any given initial data in $H^{\frac{1}{2}}(\R)$. Although the associated Cauchy problem is ill-posed in $H^s(\R)$ ($s<\frac{1}{2}$), the authors \cite{HaKiNtVi26} proved that the Cauchy problem was global in $L^2(\R)$. Precisely, there is a jointly continuous map $\Phi:\R\times L^{2}(\R)\rightarrow L^2(\R)$ that agrees with the data-to-solution map when restricted to Schwartz-class initial data. For the other results on local well posedness and global well posedness, we refer the reader to \cite{Ha21,FuHaIn17,PiTa25,TsFu81,TsFu80,Sa15,PeSaSh17,LiPoSa19,
KiNtVi23,JeLiPeSu18,JeLiPeSu20,JeLiPeSu20_global,Hao07,
WuGu17,CoKeStTaTa01,AmSi15,Tin22,HaOz94_cauchyproblem}. Although solutions are global in the energy space $H^1(\R)$ when $b=0$, the existence of blowing up solutions for $b\neq 0$ remains an open question except in half line case (see \cite{Wu13} for Dirichlet boundary condition, \cite{Tin21_halfline} for Robin boundary condition at zero).\\

When $b=0$ or $p=4$, the equation \eqref{eq1} admits a family of two-parameters solitons. Precisely, given real parameter $\omega>0$ and $-2\sqrt{\omega}<c\leq 2\sqrt{\omega}$, there exist solitons solutions of \eqref{eq1} of the form
\[
u_{\omega,c}(t,x)=e^{i\omega t}\phi_{\omega,c}(x-ct).
\] 
The profile $\phi_{\omega,c}$ are unique up to phase shift and translation (see e.g  \cite{CoOh06}) and given by an explicit formula. The stability of the solitons was proved in \cite{CoOh06} for any $(\omega,c)$ with $c^2<4\omega$. When $c=2\omega$, $c=2\sqrt{\omega}$, the associated soliton of \eqref{eq1}decay  polynomially. The orbital stability/instability of the soliton in this case remains an open question. We refer the reader to \cite{FuHa22,NiMaWu17} and the references therein for related results in this direction. If we present the scaling in the definition of the stability, the authors \cite{KwWu18} proved that the soliton is stable under an smallness condition of solutions. Later, \cite{Kim24} was success to remove the smallness condition. However, the result in \cite{Kim24} was restricted to the solutions with non-positive initial energy and momentum. We refer the reader to \cite{LiSiSu13,NiOhWu17,MiTaXu17,Guo18,Ha21,Ni20,FuHa22,Ha22,MiTaXu23,LiNi25} for the other results on stability/instability of solitons. \\

When solution is global, we study its long time behavior. The existence of multi-solitons for derivative nonlinear Schr\"odinger equations was obtained in \cite{Tin_soliton_22,Tin24_soliton}. These solutions were shown to be stable in \cite{CoWu18,TaXu18} under certain parameter conditions. The presentation of the multi-solitons implies the existence of global solutions with arbitrary large mass. Scattering theory for derivative nonlinear Schrödinger equations was investigated in \cite{BaWuXu20,HaOz94}. \\

When $b\neq 0$ and $p\neq 4$, there is no explicit formula for the soliton profile. We note that the explicit formula for solitons allows us to compute the quantities that determine the stability/instability of soliton solutions. Although the lack of an explicit formula for solitons, we have its behavior when $b>0$ sufficiently small or $p>4$ sufficiently close to $4$. To state our main result, we define soliton profile $\varphi_{\omega,c}$ by
\[
\phi_{\omega,c}(x)=\varphi_{\omega,c}(x)\exp\left(\frac{c}{2}ix-\frac{i}{4}\int_{-\infty}^x|\varphi_{\omega,c}(y)|^2 \ dy\right),
\] 
for $-2\sqrt{\omega}<c\leq 2\sqrt{\omega}$, where $\varphi_{\omega,c}$ is the unique positive, even solution to the following equation
\[
-\varphi_{xx}+\left(\omega-\frac{c^2}{4}\right)\varphi-\frac{1}{2}\Im(\varphi\overline{\varphi}_x)\varphi+\frac{c}{2}|\varphi|^2\varphi-\frac{3}{16}|\varphi|^4\varphi-b|\varphi|^p\varphi=0.
\]
In the case $b=0$, we denote the soliton profile by $\varphi_0$. When $b>0$ changes slightly, we show that the soliton profile $\varphi$ varies smoothly with respect to   parameter $b$, $p$ (see Lemma \ref{lm2.5}, \ref{lm:behave_solitons_zero mass}, \ref{lm2.6}). Consequently, we obtain the stability of solitons when $b$ and $p-4$ sufficiently small. We have the following results.
\begin{theorem}\label{thm:main_1}
Let $p\geq 4$ and $(\omega,c)$ be an admissible pair. Then, there exists $b_0=b_0(\omega,c)$ sufficiently small such that for all $b\in (0,b_0)$, the solitary wave $u_{\omega,c}(t,x)=e^{i\omega t}\phi_{\omega,c}(x-ct)$ of \eqref{eq1} is orbitally stable.
\end{theorem}

\begin{theorem}\label{thm:main_2}
Let $b\geq 0$ and $(\omega,c)$ be an admissible pair such that $-2\sqrt{\omega}<c<2s^*\sqrt{\omega}$ and $k_0=\frac{1}{\gamma}\left(2c+2\sqrt{c^2+4a\gamma}\right)>e^{\frac{1}{3}}$. Then, there exists $p_0=p_0(\omega,c)$ sufficiently close to $4$ such that for all $p \in [4,p_0)$, the solitary wave $u_{\omega,c}(t,x)=e^{i\omega t}\phi_{\omega,c}(x-ct)$ of \eqref{eq1} is orbitally stable.
\end{theorem}

\begin{remark}
The result in Theorem \ref{thm:main_1} does not contradict the well-known results in \cite{Ha22,Ha21,CoOh06,Oh14}. Indeed, Ohta \cite{Oh14} established that for $p=4$ and each $b > 0$, a unique $s^* = s^*(b) \in (0, 1)$ exists such that the soliton $u_{\omega,c}$ is orbitally stable for $c \in (-2\sqrt{\omega}, 2s^* \sqrt{\omega})$ and orbitally unstable for $c \in (2s^* \sqrt{\omega}, 2\sqrt{\omega})$. The number $s^*$ was characterized in \cite{Ha21} as the unique solution to $P(\phi_{1,2s}) = 0$ within the interval $(0,1)$. The author \cite{Ha21} also showed that $s^*(b) \rightarrow 1$ as $b\downarrow 0$. Assume that $p=4$. When b = 0, every $(\omega, c) \in \Omega$ belongs to the stability domain of the soliton (see Figure 1 in \cite{Ha21}). As $b$ increases slightly, the curve $c = 2s* \sqrt{\omega}$ remains close to $c = 2\sqrt{\omega}$, and consequently, $(\omega, c)$ stays within the stability domain. 
\end{remark}

\begin{remark}
When $b=0$ and $c=2\sqrt{\omega}$, the associated soliton of \eqref{eq1}decay  polynomially. The orbital stability/instability of the soliton in this case remains an open question. We refer the reader to \cite{FuHa22,NiMaWu17} and the references therein for related results in this direction. If we present the scaling in the definition of the stability, the authors \cite{KwWu18} proved that the soliton is stable under an smallness condition of solutions. Later, \cite{Kim24} was success to remove the smallness condition. However, the result in \cite{Kim24} was restricted to the solutions with non-positive initial energy and momentum.
\end{remark}

To prove theorem \ref{thm:main_1}, we use the result in Lemma \ref{lm2.5} and the criterion for stability of solitons obtained in \cite{CoOh06} (see Corollary \ref{coro**}). The authors \cite{CoOh06} proved that when $b=0$, $\partial_{\omega}M(\phi_{\omega,c})>0$ if $c<0$ and $\det(d''(\omega,c))=-\frac{1}{\omega}<0$ for all $(\omega,c)$ such that $c^2<4\omega$. Since the soliton profiles depend smoothly on the parameter $b$, for a sufficiently small $b$ (which may depend on $\omega$ and $c$), we have:
\[
\partial_{\omega}M(\phi_{\omega,c})>0,\quad\text{ if } c<0,
\]  
\[
\det(d''(\omega,c))<0,\quad \forall (\omega,c) \text{ such that } c^2<4\omega.
\]
Although we cannot exactly compute the above quantities, it suffices to us to obtain the orbital stability of solitons when $b$ is sufficiently small. This proves Theorem \ref{thm:main_1}. The proof of Theorem \ref{thm:main_2} is similar. \\

In the case $b=0$, the authors \cite{BaWuXu20} proved that the solution is global and scatters for small initial data in $H^s(\R)$, $\frac{1}{2}\leq s\leq 1$ when $\sigma\ge 2$, while the opposite result fails to hold when $0<\sigma<2$. When $b\neq 0$ and $p\geq 4$, the same result holds and the scattering solutions are difficult to obtain. More precisely, we have the following result. 
\begin{lemma}
\label{lm:scattering_theory_1}
Let $(\omega, c)$ be such that $4\omega>c^2$. Let $\varphi_{\omega,c}$ be the positive, even solution of \eqref{eqofvarphi} and $\phi_{\omega,c}$ be defined in \eqref{eq:varphi_change to_phi}, then
\[
\|\phi_{\omega,c}\|_{H^1(\mathbb{R})} \to 0,\quad  \text{when } c \to -2\sqrt{\omega}.
\]
\end{lemma}
The proof of Lemma \ref{lm:scattering_theory_1} relies on the implicit formula of solitons \eqref{eq_varphi_formula}, which is given in the proof of Lemma \ref{lm2.5}.\\

When $b=0$, the authors in \cite[Theorem 1.1]{HaOz94} proved that for any sufficiently small profile in the weighted Sobolev space, there exists a unique solution to \eqref{eq1} that exhibits a specific asymptotic behavior toward this profile. Moreover, solutions that scatter in the standard sense to the given profile in $H^1(\R)$ are identically zero \cite[Theorem 1.2]{HaOz94}. When $b\neq 0$ and $p\geq 4$, the presence of perturbation $b|u|^pu$ does not affect the results obtained in \cite{HaOz94}. We have the following results.
\begin{theorem} \label{thm:modified scattering operator}
Let $p\geq 4$. There exists an $\varepsilon_1 > 0$ with the following properties:\\

(1) For any $\phi_+ \in H^{4,0} \cap H^{0,4}$ with $\|\phi_+\|_{2,1} < \varepsilon_1$, \eqref{eq1} has a unique solution $u \in C(\mathbb{R}; H^{2,0}) \cap L_{\text{loc}}^4(\mathbb{R}; W^{2,\infty})$ such that for any $\alpha$ with $\frac{1}{2} < \alpha < 1$
    \begin{align*}
    (\text{MW}_+) \quad &\left\| \mathcal{G} u(t) - \exp(iS^+(t)) U(t)\phi_+ \right\|_{2,0}  \\
    &+ \left( \int_t^\infty \left\|\mathcal{G} u(\tau) - \exp(iS^+(\tau)) U(\tau)\phi_+ \right\|_{W^{2,\infty}}^4 d\tau \right)^{1/4}  \\
    &= O(t^{-\alpha}) \quad \text{as} \quad t \to \infty, 
    \end{align*}
    
(2) For any $\phi_- \in H^{4,0} \cap H^{0,4}$ with $\|\phi_-\|_{2,1} < \varepsilon_1$, \eqref{eq1} has a unique solution $u \in C(\mathbb{R}; H^{2,0}) \cap L_{\text{loc}}^4(\mathbb{R}; W^{2,\infty})$ such that for any $\alpha$ with $\frac{1}{2} < \alpha < 1$
    \begin{align*}
    (\text{MW}_-) \quad &\left\|\mathcal{G} u(t) - \exp(iS^-(t)) U(t)\phi_- \right\|_{2,0}  \\
    &+ \left( \int_{-\infty}^t \left\|\mathcal{G} u(\tau) - \exp(iS^-(\tau)) U(\tau)\phi_- \right\|_{W^{2,\infty}}^4 d\tau \right)^{1/4} \\
   & = O(|t|^{-\alpha}) \quad \text{as} \quad t \to -\infty, 
    \end{align*}
where
\[
S^{\pm}(t, x) = -\frac{1}{2} (\log |t|) \frac{x}{2|t|} \left| \hat{\phi}_{\pm} \left( \frac{x}{2t} \right) \right|^2,
\]
\[
(U(t)f)(x) = \frac{1}{\sqrt{4\pi it}} \int_{\mathbb{R}} e^{i \frac{(x-y)^2}{4t}} f(y) dy,
\]
\[
(\mathcal{G}u)(t,x)=\exp \left( - \frac{i}{2} \int_{-\infty}^x |u(t, y)|^2 dy \right) u(t, x),
\]
\[
H^{m,s} = \left\{ f \in \mathcal{S}'(\mathbb{R}) : \|f\|_{m,s} = \left\| (1 + |x|^2)^{s/2} (1 - \partial_x^2)^{m/2} f \right\|_{L^2} < \infty \right\}, \quad m, s \in \mathbb{R},
\]
and $\hat{\cdot}$ denotes the Fourier transform defined by
\[
\hat{f}(\xi) = \frac{1}{\sqrt{2\pi}} \int_{\mathbb{R}} e^{-ix\xi} f(x) dx.
\]
\end{theorem}

\begin{theorem} \label{thm:non exists non trivial scattering solutions}
Let $p\geq 4$. We assume that $u(t, x)$ is a solution of \eqref{eq1} satisfying
\begin{align*}
\makebox[1.5cm][l]{(\text{A.1})} \hfill & u(t, x) \in C(\mathbb{R}; H^{1,0}) \hfill \\
\intertext{and}
\makebox[1.5cm][l]{(\text{A.2})} \hfill & \|u(t)\|_{1,0} \le C(\|u(0)\|_{1,0}), \quad \|u(t)\|_{4} = O(|t|^{-\frac{1}{4}}) \quad \text{as} \quad t \to \infty. \hfill \\
\intertext{Moreover, we assume that there exists $\phi_+ \in H^{4,0} \cap H^{0,4}$ such that}
\makebox[1.5cm][l]{(\text{A.3})} \hfill & \|(\mathcal{G} u)(t) - U(t)\phi_+\|_{2} \to 0 \quad \text{as} \quad t \to \infty. \hfill
\end{align*}
Then $u(t, x)$ is identically zero.
\end{theorem}

\begin{remark}
Let $u$ be given as in Theorem \ref{thm:non exists non trivial scattering solutions} and $\psi=\mathcal{G}u$. From (A.3), we have, for all $x\in\R$:
\begin{align*}
&\left|\exp\left(-\frac{i}{2}\int_{-\infty}^x |u(t,y)|^2dy\right)-\exp\left(-\frac{i}{2}\int_{-\infty}^x |U(t)\phi_+|^2dy\right)\right|\\
&\leq C\left|\int_{-\infty}^x |u(t,y)|^2dy-\int_{-\infty}^x |U(t)\phi_+|^2dy\right|\\
&= C\left|\int_{-\infty}^x |\psi(t,y)|^2dy-\int_{-\infty}^x |U(t)\phi_+|^2dy\right|\\
&\leq C\norm{\psi-U(t)\phi_+}_{L^2}(\norm{\psi}_{L^2}+\norm{\phi_+}_{L^2})\rightarrow 0,
\end{align*}
as $t\rightarrow\infty$. Thus, (A.3) can be replaced by the following condition
\[
\norm{u(t)-\mathcal{G}^{-1}(U(t)\phi_+)}_2\rightarrow 0.
\]
The quantities $(MW_{+})$ and $(MW_{-})$ can be treated via an analogous approach; however, we will not pursue this detail here.
\end{remark}

The proof of Theorems \ref{thm:modified scattering operator} and \ref{thm:non exists non trivial scattering solutions} adopts the approach established in \cite{HaOz94}. We sketch their proofs in Section \ref{sec:scattering theory}. \\

The structure of the paper is outlined as follows. In Section \ref{sec:solitary waves}, we prove some important properties of solitons: existence, asymptotic behavior. In particular, for sufficiently small $b$ and $p-4$, we obtain the similar behaviors of solitons in the case $b=0$, Lemma \ref{lm2.5}, Lemma \ref{lm:behave_solitons_zero mass}, Lemma \ref{lm2.6}. These results imply the orbital stability of solitons by the criterion stated in Corollary \ref{coro**}. In Section \ref{sec:scattering theory}, we sketch the proof of Theorem \ref{thm:modified scattering operator} and Theorem \ref{thm:non exists non trivial scattering solutions}. The well-known results used in the proof of the main results    are given in Section \ref{sec:appendix}.

\section{Solitary waves}\label{sec:solitary waves}

\subsection{Existence}

A solitary wave is a solution of \eqref{eq1} of the form
\[
u_{\omega,c}(t,x)=e^{i\omega t}\phi_{\omega,c}(x-ct),
\]
where, $\phi_{\omega,c}\in H^1(\R)$ is a solution to the elliptic ordinary differential equation
\begin{equation}\label{eqofphi}
-\phi_{xx}-i|\phi|^2\phi_x+\omega\phi+ic\phi_x-b|\phi|^p\phi=0,
\end{equation}
with $4\omega\geq c^2$. Define
\begin{equation}\label{eq:varphi_change to_phi}
\phi_{\omega,c}(x)=\varphi_{\omega,c}(x)\exp\left(\frac{c}{2}ix-\frac{i}{4}\int_{-\infty}^x|\varphi_{\omega,c}(y)|^2 \ dy\right).
\end{equation}
From \eqref{eqofphi}, $\varphi_{\omega,c}\in H^1(\R)$ verifies the equation
\[
-\varphi_{xx}+\left(\omega-\frac{c^2}{4}\right)\varphi-\frac{1}{2}\Im(\varphi\overline{\varphi}_x)\varphi+\frac{c}{2}|\varphi|^2\varphi-\frac{3}{16}|\varphi|^4\varphi-b|\varphi|^p\varphi=0.
\]
Argue as in \cite{CoOh06}, we reduce the above equation to
\begin{equation}
\label{eqofvarphi}
-\varphi_{xx}+\left(\omega-\frac{c^2}{4}\right)\varphi+\frac{c}{2}|\varphi|^2\varphi-\frac{3}{16}|\varphi|^4\varphi-b|\varphi|^p\varphi=0.
\end{equation}
Define \[
a=\omega-\frac{c^2}{4}\geq 0.
\]
To prove the existence of $\varphi_{\omega,c}$, we recall the following result (see \cite{BeLi83}, \cite[Proposition 2.1]{LiTsZw21}).
\begin{lemma}\label{lm:existence solution elliptic d1}
Let $g$ be a locally Lipschitz continuous function with $g(0)=0$ and let $G(t)=\int_0^t g(s)ds$. A necessary and sufficient condition for the existence of a solution $\varphi$ of the problem
\begin{equation}\label{eq:problem varphi}
\begin{cases}
\varphi\in C^2(\R),\quad \lim_{x\rightarrow\pm\infty}\varphi(x)=0,\quad \varphi(0)>0,\\
\varphi_{xx}+g(\varphi)=0,
\end{cases}
\end{equation}
is that $k=\inf\{t>0:G(t)=0\}$ exists, $k>0$, $g(k)>0$.
\end{lemma}
Applying Lemma \ref{lm:existence solution elliptic d1}, we have the following result.
\begin{lemma}\label{lm:existence varphi}
Let $b\ge 0$. Assume that $a>0$ or ($a=0$ and $c=2\sqrt{\omega}>0$). Then, there exists a solution $\varphi$ to the problem \eqref{eq:problem varphi} with 
\[
g(t)=-at-\frac{c}{2}t^3+\frac{3}{16}t^5+bt^{p+1}.
\]
\end{lemma}
\begin{proof}
We verify that $g$ satisfies the existence conditions of Lemma \ref{lm:existence solution elliptic d1}.\\
We have
\[
G(t)=\int_0^t g(s)ds=-a\frac{t^2}{2}-\frac{c}{8}t^4+\frac{1}{32}t^6+\frac{b}{p+2}t^{p+2}=\frac{t^2}{2}H(t),
\]
with
\[
H(t)=-a-\frac{c}{4}t^2+\frac{1}{16}t^4+\frac{2b}{p+2}t^p.
\]
We prove that $H$ has only positive root on $(0,\infty)$. Indeed, we have
\[
H'(t)=\frac{-c}{2}t+\frac{1}{4}t^3+\frac{2bp}{p+2}t^{p-1}=tK(t),
\]
with 
\[
K(t)=\frac{-c}{2}+\frac{1}{4}t^2+\frac{2bp}{p+2}t^{p-2}.
\]
The function $K$ satisfies
\[
K'(t)=\frac{t}{2}+\frac{2bp(p-2)}{p+2}t^{p-3}>0, \quad \forall t>0.
\]
This implies that the function $K$ is increasing and has at most one positive root. Hence, the same holds for $H'(t)$. If $H'$ has no positive root then $H$ has at most one positive root. Otherwise, let $t_0$ be the positive root of $H'(t)$. Since $H'(0)<0$, $H'$ is negative on $(0,t_0)$ and positive on $(t_0,\infty)$. However, $H(0)=-a \leq 0$, hence, $H$ is negative on $(0,t_0)$. Thus, in both cases, $H$ has at most one positive root on $(t_0,\infty)$. Moreover, $\lim_{t\rightarrow\infty}H(t)=\infty$ and $H(t)<0$ for sufficiently small $t$. Hence, by the continuity of $H$, there exists a unique $k>0$ such that
\[
H(k)=0.
\]
Moreover,
\[
\frac{g(k)}{k}-2H(k)=a+\frac{1}{16}k^4+\frac{b(p-2)}{p+2}k^p>0,
\]
hence, $g(k)>0$. The proof is complete.
\end{proof}

Follows the argument in \cite[Lemma 2.2]{Tin24} (see also \cite[Theorem 8.1.4]{Ca03}), we have the following result.
\begin{lemma}
\label{eq:soliton profile all}
Let $b\ge 0$. Then there exists a unique positive, even solution $\varphi$ of \eqref{eqofvarphi}, which is decreasing on $\R^+$. Moreover, all nontrivial complex valued solution of \eqref{eqofvarphi} in $H^1(\R)$ are of form
\[
e^{i\theta_0}\varphi(x-x_0),
\]
for $\theta_0,x_0\in\R$. As a consequence, all $H^1(\R)$ solutions of \eqref{eqofphi} are of form
\[
e^{i\theta_0}\phi(x-x_0),
\]
where 
$$\phi(x)=\varphi(x)\exp\left(\frac{c}{2}ix-\frac{i}{4}\int_{-\infty}^x|\varphi(y)|^2dy\right).$$
\end{lemma}

\begin{proof}
The existence, positivity and evenness of $\varphi$ are established similarly to \cite[Lemma 2.2]{Tin24}. Moreover, $\varphi(0)=k$ and $\varphi_x(0)=0$. As in \cite{Tin24}, every positive solution of \eqref{eqofvarphi} is a translation of $\varphi$. Thus, there exists a unique positive, even solution of \eqref{eqofvarphi}. It remains to prove that $\varphi$ is decreasing on $\R^+$. Multiplying both sides of \eqref{eqofvarphi} by $\varphi_x$, we obtain 
\[
\frac{d}{dx}\left(\frac{-(\varphi_x)^2}{2}+\frac{a\varphi^2}{2}+\frac{c\varphi^4}{8}-\frac{\varphi^6}{32}-\frac{b\varphi^{p+2}}{p+2}\right)=0.
\]  
Since $\varphi,\varphi_x\xrightarrow[x\rightarrow\pm\infty]{} 0$, we have
\begin{equation}
\label{eq_varphi_vaphi_x^2}
-(\varphi_x)^2+a\varphi^2+\frac{c\varphi^4}{4}-\frac{\varphi^6}{16}-\frac{2b\varphi^{p+2}}{p+2}=0.
\end{equation}
Next, we show that $\varphi_x$ is non-vanishing on $\R^+$. Arguing by contradiction, suppose that there exists $x_0>0$ such that $\varphi_x(x_0)=0$. Then, by \eqref{eq_varphi_vaphi_x^2}, $\varphi(x_0)$ is a positive solution of $H(t)=0$. This implies that $\varphi(x_0)=k$ since $H$ has only one positive root on $(0,\infty)$. Therefore, $\varphi(x)=\varphi(x+x_0)$ for all $x>0$. It follows that $\varphi$ is a periodic function of period $x_0$, which is a contradiction. In conclusion, $\varphi_x(x)\neq 0$, for all $x>0$. Moreover, $\varphi\xrightarrow[x\rightarrow\infty]{} 0$, hence, $\varphi_x<0$ on $\R^+$. The remainder of the proof follows similarly to \cite[Lemma 2.2]{Tin24}. 

\end{proof}

\subsection{Asymptotic behavior of soliton profile with respect to parameter $b$}

In this section, we study the behavior of the solitary wave profile i.e the positive solution of \eqref{eqofvarphi}. Let $\varphi$ be the positive solution of \eqref{eqofvarphi} given in Lemma \ref{lm:existence varphi}. From \eqref{eq_varphi_vaphi_x^2} and $\varphi_x<0$ for all $x>0$, we have
\begin{equation}\label{eqofvarphi2}
\varphi_x=-\varphi\sqrt{a+\frac{c}{4}\varphi^2-\frac{1}{16}\varphi^4-\frac{2b}{p+2}\varphi^p},\quad\forall x>0.
\end{equation}
Since $\varphi$ is even and decreasing on $(0,\infty)$, $\varphi$ attains its maximum at zero. \\
When $b=0$, it is proved that there exists a unique positive, even solution $\varphi_0$ of \eqref{eqofvarphi2}. Moreover,
\begin{equation}\label{eqofvarphi0}
\varphi_0^2(x)=
\begin{cases}
\frac{8a}{\sqrt{c^2+4a}\cosh(\sqrt{4a}x)-c},\quad\text{ if } c^2<4\omega,\\
\frac{4c}{(cx)^2+1},\quad\text{ if } c=2\sqrt{\omega}.
\end{cases}
\end{equation}
Thus, if $c^2<4\omega$ then
\[
\varphi_0(x)\approx e^{-\sqrt{a} |x|},\quad \forall x\in\R.
\]
For convenience, we define
\begin{align*}
D(\varphi)&=a+\frac{c}{4}\varphi^2-\frac{1}{16}\varphi^4-\frac{2b}{p+2}\varphi^p.
\end{align*}
When $b>0$, there is no explicit formula of $\varphi$. However, we prove that $\varphi$ admits a similar asymptotic behavior as $\varphi_0$. More precisely, $\varphi$ exhibits the same asymptotic behavior as the solution to the linear ODE
$$\varphi_x=-\sqrt{a} \varphi.$$
Although the proof could be extended to a more general function $D$, doing so falls outside the scope of this paper. 

\begin{lemma}\label{eq_behavior_of_varphi}
For all $x>0$, we have
\[
\varphi(x)\approx e^{-\sqrt{a} x},
\]
i.e there exist $C_1,C_2>0$ such that
\begin{equation}
\label{eq:ass_beha_varphi}
C_1e^{-\sqrt{a} x}\le \varphi(x)\le C_2e^{-\sqrt{a} x},\quad\forall x>0.
\end{equation}
\end{lemma}
\begin{proof}
Since $\varphi\xrightarrow[x\rightarrow\pm\infty]{} 0$, for $x>0$ large enough, we have 
\begin{equation}
\label{eq:estimate_of_nonlinear term_ODE}
(\sqrt{a}-\varphi)^2<D(\varphi)=a+\frac{c}{4}\varphi^2-\frac{1}{16}\varphi^4-\frac{2b}{p+2}\varphi^p<(\sqrt{a}+c_0 \varphi)^2,
\end{equation} 
where $c_0=\frac{\sqrt{|c|}}{2}$. Combining to \eqref{eqofvarphi2}, we have, for $x$ large enough,
\begin{equation}\label{eq_estimate_of_varphi_x}
-\varphi(\sqrt{a}-\varphi)>\varphi_x>-(\sqrt{a}+c_0\varphi)\varphi.
\end{equation}
This implies that
\[
\partial_x[\text{ln}(c_0\varphi)-\text{ln}(\sqrt{a}+c_0\varphi)]>-\sqrt{a}.
\]
Thus, for $x>M$ large enough, by integrating both sides from $M$ to $x$, we get that
\[
\varphi \approx \frac{c_0\varphi}{\sqrt{a}+c_0\varphi}>e^{-\sqrt{a}(x-M)+C(M)}.
\]
Hence, there exists $C_1>0$ such that
\begin{equation}\label{eq_prove_estimate_varphi_1}
\varphi(x)>C_1e^{-\sqrt{a}x}
\end{equation}
Moreover, by \eqref{eq_estimate_of_varphi_x}, we have, for $x>M$ large enough,
\begin{equation}\label{eq_estimate ln_varphi}
\frac{d}{dx}\ln \varphi<-\sqrt{a}+\varphi<\frac{-\sqrt{a}}{2}.
\end{equation}
It implies that 
\[
\ln\varphi (x)-\ln\varphi(M)<\frac{-\sqrt{a}}{2} (x-M).
\]
It is equivalent to 
\[
\varphi(x)<C(M)e^{-\frac{\sqrt{a}}{2}x},
\]
for $x>M$ and $C(M)$ denotes a positive constant depending on $M$. Combining to \eqref{eq_estimate ln_varphi}, we have
\begin{align*}
\ln\varphi(x)-\ln\varphi(M) &<-\sqrt{a}(x-M)+C(M)\int_M^x e^{-\frac{\sqrt{a}}{2}y}dy<-\sqrt{a}(x-M)+2C(M)\frac{e^{-\frac{\sqrt{a}}{2}M}-e^{-\frac{\sqrt{a}}{2}x}}{\sqrt{a}}\\
&<-\sqrt{a}(x-M)+C(M).
\end{align*}
Hence, for $x>M$, 
\[
\varphi(x)<C(M)e^{-\sqrt{a}x}.
\]
Combining to \eqref{eq_prove_estimate_varphi_1}, the proof is complete.
\end{proof}

Next, we compare $\varphi$ and $\varphi_0$, which is necessary for us to obtain a more explicit stability result for $\varphi$. We have the following result.
\begin{lemma}\label{lm2.5}
The function $\varphi$ satisfies the following properties:
\begin{equation}\label{eq_need to prove lm}
\norm{\varphi-\varphi_0}_{L^{\infty}}=O(\sqrt[4]{b}),\quad \norm{\varphi-\varphi_0}_{L^2}=O(\sqrt[4]{b}).
\end{equation}
Moreover, 
\begin{equation}
\label{eq_need_prove2}
\norm{\partial_{\omega}\varphi-\partial_{\omega}\varphi_0}_{L^2}=O(\sqrt[4]{b}),
\end{equation}
and
\begin{equation}
\label{eq_need_prove3}
\norm{\partial_c\varphi-\partial_c\varphi_0}_{L^2}=O(\sqrt[4]{b}).
\end{equation}
\end{lemma}

\begin{proof}
Throughout the proof, $O(b)$ denotes any function uniformly bounded by $Cb$, where $C>0$ is a constant independent of $b$. We divide the proof in several steps.

\textbf{Step 1: Formulating the profile $\varphi$.}\\ 
Let $y=\varphi^2$. As $\varphi$ is an even function, so is $y$. Multiply both sides of \eqref{eqofvarphi2} by $\varphi$, we have
\begin{equation}\label{eq_of_y}
y'=-y\sqrt{4a+cy-\frac{1}{4}y^2-\frac{8b}{p+2}y^{p/2}}.
\end{equation}
For convenience, let 
\[
F(y)=4a+cy-\frac{1}{4}y^2,\quad F_b(y)=F(y)-\frac{8b}{p+2}y^{p/2}.
\]
We have $$F(y)=\frac{1}{4}(y_+ - y)(y-y_{-}),$$ with $y_{\pm}=2c\pm 2\sqrt{c^2+4a}$. It is easy to verify that $y_+>0>y_{-}$. Moreover, $F(y(x))>F_b(y(x))\ge 0$ for all $x\in\R$. Thus, $y_+ > y(x)> 0$ for all $x\in\R$. In particular, $y_+>y(0)$. It implies that $F(y(x))>\frac{1}{4}(y_+ - y(0))(-y_{-})>0$ for all $x\in\R$. 

Let $k_b$ be the maximal value of $y$ on $\R$ i.e $k_b=y(0)$. When $b=0$, let $y_0=\varphi_0^2$ and the maximal value of $y_0$ on $\R$
\[
k_0=y_0(0)=y_+,
\]  
We have
\[
F_b(k_0)=F(k_0)-\frac{8b}{p+2}k_0^{p/2}=-\frac{8b}{p+2}k_0^{p/2}<0.
\]
Moreover, $F_b(t)>0$ for all $t\in (0,k_b)$. This implies that $k_0>k_b$. 

Since $k_b$ is the first positive root of $F_b$, we have
\[
0=F_b(k_b)=F(k_b)-\frac{8b}{p+2}k_b^{p/2}= \frac{1}{4}(y_+ - k_b)(k_b-y_{-})-\frac{8b}{p+2}k_b^{p/2}>\frac{1}{4}(k_0-k_b)(-y_{-})-\frac{8b}{p+2}k_0^{p/2}.
\]
This implies that
\begin{equation}
\label{eq:estimate k_0-k_b}
0<k_0-k_b<\frac{4}{|y_{-}|}\frac{8b}{p+2}k_0^{p/2}=O(b),
\end{equation}
%and for $b>0$ small enough,
%\begin{align*}
%F(y)&>\frac{1}{4}(y_+ - y(0))(-y_{-})=\frac{1}{4}(y_{+}-k_b)(-y_{-})\\
%&=\frac{8b}{p+2}k_b^{p/2}\frac{-y_{-}}{k_b-y_{-}}>\frac{8b}{p+2}(k_0/2)^{p/2}\frac{-y_{-}}{k_0-y_{-}}.
%\end{align*}

From \eqref{eq_of_y}, we have
\begin{equation}\label{eq_of_y_2}
y'=-y\sqrt{F_b(y)}=-y\sqrt{F(y)}+\frac{8b}{(p+2)(\sqrt{F(y)}+\sqrt{F_b(y)})} y^{1+p/2}.
\end{equation}

Thus, dividing both sides by $y\sqrt{F(y)}$ and integrating from $0$ to $x$, we obtain
\begin{equation}\label{eq_varphi_o(b)}
\int_0^x \frac{y'}{y\sqrt{F(y)}} dt= \int_0^x -1+\frac{8b y^{p/2}}{(p+2)\sqrt{F(y)}(\sqrt{F(y)}+\sqrt{F_b(y)})}dt,\quad \forall x>0.
\end{equation}
By changing variable $y=\frac{y_+}{1+z}$, we have
\begin{align}
\int \frac{y'}{y\sqrt{F(y)}} dt &=2\int \frac{dy}{y\sqrt{(y_+ - y)(y-y_{-})}}\nonumber\\
&=-2\int \frac{dz}{\sqrt{y_+ z (y_+ - y_{-}- y_{-} z)}}\nonumber\\
&=-2\int \frac{dz}{\sqrt{16a((z+C_0)^2-C_0^2)}},\quad (\text{ with } 32aC_0=y_+(y_+ -y_{-})),\nonumber\\
&=\frac{-1}{2\sqrt{a}}\int \frac{d\hat{z}}{\sqrt{\hat{z}^2-C_0^2}},\quad (\text{ with } \hat{z}=z+C_0),\nonumber\\
&=\frac{-1}{2\sqrt{a}}\int \frac{dw}{\sqrt{w^2-1}},\quad (\text{ with } \hat{z}=C_0w),\nonumber\\
&=\frac{-1}{2\sqrt{a}}\text{arcosh}(w).\label{eq_arcoshw}
\end{align}
Note that from inequality $y_+> y(x)> y_{-}$, it follows that $$w(x)=\frac{y_+}{C_0 y(x)}+\frac{C_0-1}{C_0}>1,\quad \text{ for all } x\in\R.$$ 
Therefore, the function $\text{arcosh}(w)$ is well-defined on $\R$. \\
Combining \eqref{eq_arcoshw} with \eqref{eq_varphi_o(b)} and $y=\varphi^2$, we have 
\[
\frac{-1}{2\sqrt{a}}\left(\text{arcosh}\left(\frac{y_+}{C_0\varphi^2(x)}+\frac{C_0-1}{C_0}\right)-\text{arcosh}(w(0))\right)=-x+F_1(x),
\]
where,
\[
F_1(x)=\int_0^x \frac{8by^{p/2}}{(p+2)\sqrt{F(y)}(\sqrt{F(y)}+\sqrt{F_b(y)})}dt>0, \quad \forall x>0.
\]
This implies that
\[
\varphi^2(x)=\frac{1}{\frac{C_0}{y_+}\cosh(2\sqrt{a}x+\text{arcosh}(w(0))-2\sqrt{a}F_1(x))-\frac{C_0-1}{y_+}}.
\]
Substituting $C_0=\frac{y_+(y_+ - y_{-})}{32a}$, $y_{\pm}=2c\pm 2\sqrt{c^2+4a}$ into the above expression, we get that
\begin{equation}\label{eq_varphi_formula}
\varphi^2(x)=\frac{8a}{\sqrt{c^2+4a} \cosh(2\sqrt{a}x+\text{arcosh}(w(0))-2\sqrt{a}F_1(x))-c}.
\end{equation}

\textbf{Step 2:Estimate $F_1$.}\\
Since $y$ is an even function, $F_1$ is odd. Thus, it suffices to estimate $F_1$ on $(0,\infty)$. Given two functions $f,g$, we write $f\approx g$ if there exist two positive constants $C_1,C_2$ such that $C_1f(x)\leq g(x)\leq C_2f(x)$ for all $x\in\R$. Noting that $F(y)\approx (k_0 -y)(y-y_{-})\approx k_0-y$, with the implicit constants are independent of $b$, we have
\begin{equation}
\label{eq:approx F_1}
F_1(x)\approx b\int_0^x \frac{y^{p/2}}{F(y)}dt  \approx b\int_0^x\frac{y^{p/2}}{k_0-y}dt.
\end{equation}
From Lemma \ref{eq_behavior_of_varphi}, when $t>T$ (with $T$ is large enough, does not depend on $b$), $y(t)\approx 0$ and 
\[
\frac{y^{p/2}}{k_0-y}\approx y^{p/2}.
\] 
Thus, 
\[
b\int_T^{\infty}\frac{y^{p/2}}{k_0-y}dt=O(b),
\]
so
\[
\sup_{x\in\R}|F_1(x)| \leq O(b)+b\int_0^T \frac{y^{p/2}}{k_0-y} dt.
\]
We prove that 
\begin{equation}
\label{eq:k_0-y}
k_0-y(t) \geq Ab+Bt^2, \quad \forall t\in [0,T].
\end{equation}
We have
\begin{equation*}
k_0-y(t)=k_0-k_b+y(0)-y(t)\geq Ab + y(0)-y(t).
\end{equation*}
Hence, \eqref{eq:k_0-y} reduces to
\begin{equation}
\label{eq:y(0)_y(t)}
y(t) \leq y(0)- Bt^2, \quad \forall t\in [0,T].
\end{equation}
We have
\[
F'(y_+) = \frac{1}{4}(y_{-}-y_{+})=-\lambda<0.
\]
Then, there exists a sufficiently small $\delta > 0$ that is independent of $b$ such that
\[
F'(y)<\frac{-3\lambda}{4}, \quad \forall y\in [k_0-\delta,k_0].
\]
Therefore, $\forall y\in [k_0-\delta,k_b]$, for sufficiently small $b$,
\[
F_b'(y)<F'(y)+|F_b'(y)-F'(y)|\leq -\frac{3\lambda}{4}+ \frac{4pb}{p+2}k_0^{\frac{p}{2}-1}<-\frac{\lambda}{2}.
\]
Let $t_0>0$ be the unique positive number such that $y(t_0)=k_0-\delta$. It follows that $k_b\geq y(t)\geq k_0-\delta,\ \forall t\in [0,t_0]$ and $k_0-\delta\geq y(t)> 0,\ \forall t\in [0,\infty)$.\\
 
On $[0,t_0]$, since $F_b' \leq -\frac{\lambda}{2} < 0$, applying Lagrange theorem, we have:
\begin{equation*}
    F_b(y)=F_b(y)-F_b(k_b)=F_b'(\xi)(y-k_b) \geq \frac{\lambda}{2}(k_b - y)
\end{equation*}
Using the differential equation $y' = -y\sqrt{F_b(y)}$, we obtain:
\begin{equation*}
    \frac{-y'}{\sqrt{k_b - y}} \geq (k_0-\delta)\sqrt{\frac{\lambda}{2}}. 
\end{equation*}
Integrating both sides from $0$ to $t$ ($t \le t_0$):
\begin{equation*}
    2\sqrt{k_b - y(t)} \geq (k_0-\delta)\sqrt{\frac{\lambda}{2}} t \implies k_b - y(t) \ge (k_0-\delta)^2\frac{\lambda}{8} t^2
\end{equation*}
\begin{equation*}
    \implies y(t) - y(0) \leq -\frac{\lambda (k_0-\delta)^2}{8} t^2 \quad \forall t \in [0, t_0]
\end{equation*}
For $t > t_0$, the function $y(t)$ drops below the level $y_+ - \delta$. Consequently, the distance from $y(t)$ to the peak $k_b$ is strictly bounded below by a positive constant:
\begin{equation*}
    k_b - y(t) \ge k_b - (y_+ - \delta) \ge \delta - O(b)
\end{equation*}
For a sufficiently small $b$, it follows that $k_b - y(t) \ge \frac{\delta}{2} > 0$.

Furthermore, since $t \in [t_0, T]$, the quantity $t^2$ is bounded above by $T^2$. This allows for the introduction of the factor $t^2$ as follows:
\begin{equation*}
    k_b - y(t) \ge \frac{\delta}{2} = \left(\frac{\delta}{2T^2}\right) T^2 \ge \left(\frac{\delta}{2T^2}\right) t^2
\end{equation*}
Rearranging terms yields the desired inequality:
\begin{equation*}
    y(t) - k_b \le -\left(\frac{\delta}{2T^2}\right) t^2 \quad \forall t \in [t_0, T].
\end{equation*}
Letting $B=\min\left\{\frac{\lambda (k_0-\delta)^2}{8},\frac{\delta}{2T^2}\right\}$, \eqref{eq:y(0)_y(t)} is proved. 

From \eqref{eq:k_0-y} and the upper boundedness of $y(t)$, we have
\begin{equation*}
    \int_0^T \frac{y^{p/2}(t)}{k_0 - y(t)} dt \leq \int_0^T \frac{k_0^{p/2}}{A  b + B t^2} dt
\end{equation*}
Multiplying by the factor $b$ yields:
\begin{equation*}
    L = b \int_0^T \frac{y^{p/2}(t)}{k_0 - y(t)} dt \leq k_0^{p/2}  b \int_0^T \frac{1}{A  b + B t^2} dt
\end{equation*}

%To evaluate the integral on the right-hand side, we employ the substitution:
%\begin{equation*}
%    t = \sqrt{\frac{A b}{B}}  u \implies dt = \sqrt{\frac{A  b}{B}} du
%\end{equation*}
%The integration bounds are transformed as follows:
%\begin{itemize}
%    \item For $t = 0 \implies u = 0$
%    \item For $t = T \implies u = T\sqrt{\frac{B}{A  b}}$
%\end{itemize}
%Substituting these into the integrand leads to:
Using the substitution $t = \sqrt{\frac{A b}{B}} u$, we have
\begin{align*}
    \int_0^T \frac{1}{A  b + B t^2} dt &= \int_0^{T\sqrt{\frac{B}{A  b}}} \frac{1}{A  b + A  b  u^2} \left( \sqrt{\frac{A  b}{B}} du \right) \\
    &= \frac{1}{A  b} \sqrt{\frac{A  b}{B}} \int_0^{T\sqrt{\frac{B}{A  b}}} \frac{1}{1 + u^2} du = \frac{1}{\sqrt{A B  b}} \left[ \arctan(u) \right]_0^{T\sqrt{\frac{B}{A  b}}} \\
    &= \frac{1}{\sqrt{A B  b}} \arctan\left(T\sqrt{\frac{B}{A  b}}\right)
\end{align*}
Plugging this result back into the upper bound of $L$, we obtain:
\begin{align*}
    L &\leq k_0^{p/2}  b  \left( \frac{1}{\sqrt{A B \cdot b}} \arctan\left(T\sqrt{\frac{B}{A  b}}\right) \right) \\
      &\leq \frac{k_0^{p/2}}{\sqrt{A B}}  \sqrt{b}  \arctan\left(T\sqrt{\frac{B}{A \cdot b}}\right)\leq \frac{k_0^{p/2}}{\sqrt{A B}}  \sqrt{b} \frac{\pi}{2}=O(\sqrt{b}).
\end{align*}
This implies that 
\[
\sup_{x\in\R}|F_1(x)|\leq O(b)+O(\sqrt{b})=O(\sqrt{b}).
\]
Hence, $F_1=O(\sqrt{b})$.

\textbf{Step 3: Estimate $w(0)$ and prove \eqref{eq_need to prove lm}.}\\

We have $w(0)>1$ and
\[
w(0)=\frac{y_+}{C_0y(0)}+\frac{C_0-1}{C_0}=\frac{k_0}{C_0 k_b}+\frac{C_0-1}{C_0}=1+O(b).
\]
It follows that, for $b$ small enough,
\begin{align*}
\text{arcosh}(w(0))&=\ln (w(0)+\sqrt{w(0)^2-1})=\ln(1+O(b)+\sqrt{(1+O(b))^2-1})\\
&=\ln(1+O(b))=O(b).
\end{align*}
Substituting into \eqref{eq_varphi_formula}, we obtain
\[
\varphi^2(x)=\frac{8a}{\sqrt{c^2+4a} \cosh(2\sqrt{a}x+O(\sqrt{b}))-c}.
\]
We have
\begin{align*}
\cosh(z+O(\sqrt{b}))&=(1+O(\sqrt{b}))\cosh(z),\\
\sqrt{c^2+4a}(1+O(\sqrt{b}))&\ge |c|+\varepsilon,
\end{align*}
for some $\varepsilon>0$ provided that $b$ is sufficiently small. Thus, by letting $z=2\sqrt{a}x$, 
\begin{align*}
|\varphi(x)-\varphi_0(x)|^2 &\le |\varphi^2(x)-\varphi_0^2(x)|=8a\left|\frac{1}{\sqrt{c^2+4a}(1+O(\sqrt{b})) \cosh(z)-c}-\frac{1}{\sqrt{c^2+4a} \cosh(z)-c}\right|\\
&=8aO(b)\sqrt{c^2+4a}\left|\frac{\cosh(z)}{(\sqrt{c^2+4a}(1+O(\sqrt{b})) \cosh(z)-c)(\sqrt{c^2+4a} \cosh(z)-c)}\right|\\
&\le 8aO(\sqrt{b})\sqrt{c^2+4a}\frac{1}{\left((|c|+\varepsilon) \cosh(z)-c\right)(\sqrt{c^2+4a}-c)}\\
&\leq\frac{O(\sqrt{b})}{(|c|+\varepsilon) \cosh(z)-c}.
\end{align*}
This easily implies the first and the second estimates in \eqref{eq_need to prove lm}. It remains to prove \eqref{eq_need_prove2}. \\

\textbf{Step 3: Prove \eqref{eq_need_prove2}.}\\
We have
\begin{equation}\label{eq_of_var_phi}
-\varphi_{xx}+a\varphi+\frac{c}{2}\varphi^3-\frac{3}{16}\varphi^5-b\varphi^{p+1}=0.
\end{equation}
Differentiating both sides with respect to $\omega$ yields
\begin{align*}
0&=\left(-\partial_{xx}+a+\frac{3c}{2}\varphi^2-\frac{15}{16}\varphi^4-b(p+1)\varphi^p\right)\partial_{\omega}\varphi+\varphi=L_b\partial_{\omega}\varphi+\varphi,
\end{align*}
where $$L_b=-\partial_{xx}+a+\frac{3c}{2}\varphi^2-\frac{15}{16}\varphi^4-b(p+1)\varphi^p.$$ 
In particular, by let 
\[
L_0=-\partial_{xx}+a+\frac{3c}{2}\varphi_0^2-\frac{15}{16}\varphi_0^4,
\]
then $L_0\partial_{\omega}\varphi_0=-\varphi_0$. This implies that
\begin{equation}\label{eq:L_0_1}
L_0(\partial_{\omega}\varphi-\partial_{\omega}\varphi_0)=(L_0-L_b)\partial_{\omega}\varphi+L_b\partial_{\omega}\varphi-L_0\partial_{\omega}\varphi_0=(L_0-L_b)\partial_{\omega}\varphi-(\varphi-\varphi_0).
\end{equation}
Furthermore, according to Lemma 3.4 in \cite{LiSiSu13}, the operator $L_0$ possesses exactly one negative eigenvalue, its kernel is given by $\text{Ker}(L_0)=\text{Span}\{\partial_x\varphi_0\}$, and the remainder of its spectrum is bounded below by a positive constant. Since $\partial_x\varphi_0$ is an odd function, whereas $\varphi$, $\varphi_0$, $\partial_{\omega}\varphi$, and $\partial_{\omega}\varphi_0$ are even, these functions all belong to $\text{Ker}(L_0)^{\perp}$. Consequently, both $\partial_{\omega}\varphi-\partial_{\omega}\varphi_0$ and the right-hand side of \eqref{eq:L_0_1} also belong to $\text{Ker}(L_0)^{\perp}$.\\

Let $\hat{L}_0=L_0\mid_{\text{Span}\{\partial_x\varphi_0\}^{\perp}}$. Then $\hat{L}_0$ is a self-adjoint, invertible operator whose inverse $\hat{L}_0^{-1}$ is a bounded linear operator. Moreover , its spectrum satisfies
$$\sigma(\hat{L}_0^{-1}) = \left\{ \frac{1}{\lambda} \ \middle|\ \lambda \in \sigma(\hat{L}_0) \right\}.$$
By Gelfand's formula \ref{lm:Gelfand formula} and Lemma \ref{lm:Gel_for_adjoint}, we have
$$\rho(\hat{L}_0^{-1})=\|\hat{L}_0^{-1}\|_{L^2}.$$
Furthermore, since the spectrum of $\hat{L}_0$ is bounded away from zero, $\rho(\hat{L}_0^{-1})$ can be explicitly evaluated as
\[
\rho(\hat{L}_0^{-1})=\sup_{\lambda\in \sigma(\hat{L}_0^{-1})}|\lambda|=\sup_{\lambda\in \sigma(\hat{L}_0)}\frac{1}{|\lambda|}=\frac{1}{\inf_{\lambda\in \sigma(\hat{L}_0)}|\lambda|}=\delta>0.
\]
Hence, from \eqref{eq:L_0_1} and \eqref{eq_need to prove lm}, we have
\begin{align*}
\norm{\partial_{\omega}\varphi-\partial_{\omega}\varphi_0}_{L^2}&=\norm{\hat{L}_0^{-1}\left[(L_0-L_b)\partial_{\omega}\varphi-(\varphi-\varphi_0)\right]}_{L^2}\leq \delta \norm{(L_0-L_b)\partial_{\omega}\varphi-(\varphi-\varphi_0)}_{L^2}\\
&\leq\delta \|(L_0-L_b)(\partial_{\omega}\varphi-\partial_{\omega}\varphi_0)\|_{L^2}+\delta\|(L_0-L_b)\partial_{\omega}\varphi_0\|_{L^2}+\delta\|\varphi-\varphi_0\|_{L^2}\\
&\leq \delta\left\|\frac{3c}{2}(\varphi_0^2-\varphi^2)-\frac{15}{16}(\varphi_0^4-\varphi^4)+b(p+1)\varphi^p\right\|_{L^{\infty}}\left(\norm{\partial_{\omega}\varphi-\partial_{\omega}\varphi_0}_{L^2}+\norm{\partial_{\omega}\varphi_0}_{L^2}\right)+\delta\|\varphi-\varphi_0\|_{L^2}\\
&\leq O(\sqrt[4]{b})\norm{\partial_{\omega}\varphi-\partial_{\omega}\varphi_0}_{L^2}+O(\sqrt[4]{b}).
\end{align*}
This implies that, for $b$ sufficiently small,
\[
\norm{\partial_{\omega}\varphi-\partial_{\omega}\varphi_0}_{L^2}=O(\sqrt[4]{b}).
\]
This proves \eqref{eq_need_prove2}. It remains to prove \eqref{eq_need_prove3}. \\
Differentiating both sides of \eqref{eq_of_var_phi} with respect to $c$, we obtain
\[
0=\left(-\partial_{xx}+a+\frac{3c}{2}\varphi^2-\frac{15}{16}\varphi^4-b(p+1)\varphi^p\right)\partial_c\varphi+\left(-\frac{c}{2}\varphi+\frac{1}{2}\varphi^3\right)=L_b\partial_c\varphi-\frac{c}{2}\varphi+\frac{1}{2}\varphi^3.
\]
In particular, $L_0\partial_c\varphi_0-\frac{c}{2}\varphi_0+\varphi_0^3=0$. \\
Therefore,
\[
L_0(\partial_c\varphi-\partial_c\varphi_0)=(L_0-L_b)\partial_c\varphi+L_b\partial_c\varphi-L_0\partial_c\varphi_0=(L_0-L_b)\partial_c\varphi+\frac{c}{2}(\varphi-\varphi_0)-(\varphi^3-\varphi_0^3).
\]
This implies \eqref{eq_need_prove3} holds by similar argument to prove \eqref{eq_need_prove2}. The proof is complete.

\end{proof}

\subsection{Asymptotic behavior of soliton profile in borderline case}

Recall that when $b=0$ and $c=2\sqrt{\omega}$ we have
\[
\varphi_0^2(x)=\frac{4c}{(cx)^2+1}.
\]
Inspired by the above formulation, we investigate the behavior of solitons to \eqref{eq1} under the conditions $b\neq 0$, $c=2\sqrt{\omega}$ (hence, $a=0$), and $p>4$. We have the following results.

\begin{lemma}
\label{lm:behave_varphi_zero mass}
As $|x|\rightarrow\infty$, $\varphi$ decays to zero at a polynomial rate. More precisely,
\[
\varphi(x) \approx (1+|x|)^{-1}.
\] 
\end{lemma}

\begin{proof}
The proof is similar to the one in \cite[Lemma 2.3]{Tin24}. By evenness of $\varphi$, it suffices to consider $x\in (0,\infty)$. From \eqref{eqofvarphi2}, we have, for $x>0$ large enough,
\[
\varphi_x \approx -\varphi^2,
\]
and hence,  
\[
\partial_x \left(\frac{1}{\varphi}\right)= \frac{-\varphi_x}{\varphi^2}\approx 1.
\]
This implies the desired result. 
\end{proof}

\begin{lemma}
\label{lm:behave_solitons_zero mass}
The soliton profile $\varphi$ changes regularly with respect to perturbation $b|u|^pu$. More precisely, we have
\begin{equation}\label{eq_prove_behave_solitons_zero_mass}
\norm{\varphi-\varphi_0}_{L^{\infty}}=O(\sqrt[4]{b}),\quad \norm{\varphi-\varphi_0}_{L^2}=O(\sqrt[4]{b}).
\end{equation}
\end{lemma}

\begin{proof}
We follow the lines in the proof of Lemma \ref{lm2.5}. We divide the proof in several steps.
\textbf{Step 1: Formulating the profile $\varphi$.}\\
Since $\varphi$ is an even function, we only need to focus on $[0,\infty)$. By Lemma \ref{eq:soliton profile all}, $\varphi$ is positive and decreasing.\\
From \eqref{eqofvarphi2}, we have
\begin{equation}\label{eq_varphi_zero mass_0}
\varphi'=-\varphi^2\sqrt{\frac{c}{4}-\frac{1}{16}\varphi^2-\frac{2b}{p+2}\varphi^{p-2}},\quad\forall x>0.
\end{equation}
For convenience, let
\[
G(y)=\frac{c}{4}-\frac{1}{16}y^2,\quad G_b(y)=G(y)-\frac{2b}{p+2}y^{p-2}.
\]
We have
\[
\varphi'=-\varphi^2\sqrt{\frac{c}{4}-\frac{1}{16}\varphi^2}+\frac{2b\varphi^p}{(p+2)(\sqrt{G(\varphi)}+\sqrt{G_b(\varphi)})}.
\]
Thus, dividing both sides by $\varphi^2\sqrt{\frac{c}{4}-\frac{1}{16}\varphi^2}$ and integrating from $0$ to $x$, we obtain
\begin{equation}\label{eq_varphi_zero mass}
\int_0^x \frac{\varphi'}{\varphi^2\sqrt{\frac{c}{4}-\frac{1}{16}\varphi^2}}dt=\int_0^x -1+\frac{2b\varphi^{p-2}}{(p+2)\sqrt{G(\varphi)}(\sqrt{G(\varphi)}+\sqrt{G_b(\varphi)})}dt=-x+G_1(x),
\end{equation}
where 
\[
G_1(x)=\int_0^x \frac{2b\varphi^{p-2}}{(p+2)\sqrt{G(\varphi)}(\sqrt{G(\varphi)}+\sqrt{G_b(\varphi)})}dt>0.
\]
Let $h_b=\varphi(0)$ and $X=\varphi(x)$. The left hand side of \eqref{eq_varphi_zero mass} equals to
\begin{equation}\label{eqof_z}
LHS=\int_{h_b}^{X}\frac{dz}{z^2\sqrt{\frac{c}{4}-\frac{1}{16}z^2}}=\int_{h_b}^{X}\frac{4dz}{z^2\sqrt{4c-z^2}}.
\end{equation}
Noting that from \eqref{eq_varphi_zero mass_0}, we have $G(\varphi(0))>G_b(\varphi(0))\geq 0$. This yields $4c>\varphi^2(0)=h_b^2>\varphi^2(x)$ for all $x\in\R^+$. Thus, the integrand in \eqref{eqof_z} is well-defined. \\
By changing variables $z=2\sqrt{c}\sin(t)$, $t\in \left(0,\frac{\pi}{2}\right)$, we have
\begin{align*}
\int \frac{4dz}{z^2\sqrt{4c-z^2}}&=\frac{1}{c}\int \frac{1}{\sin^2(t)}dt=\frac{-1}{c}\cot(t)+C=\frac{-\sqrt{4c-z^2}}{cz}+C.
\end{align*}
Thus,
\[
LHS=\frac{\sqrt{4c-h_b^2}}{ch_b}-\frac{\sqrt{4c-\varphi(x)^2}}{c\varphi(x)}.
\]
Subtitling to \eqref{eq_varphi_zero mass}, we obtain
\[
\varphi^2(x)=\frac{4c}{1+c^2\left(x-G_1(x)+\frac{\sqrt{4c-h_b^2}}{ch_b}\right)^2}.
\]
\textbf{Step 2. Estimate $\frac{\sqrt{4c-h_b^2}}{ch_b}$.} \\
Lemma \ref{eq:soliton profile all} implies that $h_b$ is the first positive root of $G_b$. Moreover, it is clear that $h_0:=\sqrt{4c}>h_b$ is the first positive root of $G$. Therefore,
\[
0=G_b(h_b)=G(h_b)-\frac{2b}{p+2}h_b^{p-2}=\frac{1}{16}(h_0^2-h_b^2)-\frac{2b}{p+2}h_b^{p-2}.
\]
Consequently,
\[
0<h_0-h_b=\frac{32bh_b^{p-2}}{(p+2)(h_0+h_b)}<\frac{32bh_0^{p-2}}{(p+2)h_0}=O(b),
\]
and hence, for sufficiently small $b$,
\begin{equation}
\label{bounded constant by O(b)}
\frac{\sqrt{4c-h_b^2}}{ch_b}=\frac{\sqrt{h_0^2-h_b^2}}{ch_b}<\frac{\sqrt{(h_0-h_b)2h_0}}{ch_0/2}=O(\sqrt{b}).
\end{equation}
\textbf{Step 3. Estimate $G_1$.} We follow the proof of step 2 of Lemma \ref{lm2.5}. Since $G_1$ is an odd function, it suffices to estimate $G_1$ on $\R^+$.\\
We have 
\begin{equation}\label{eq_estimate G_1}
G_1(x) \leq\int_0^x\frac{2b\varphi^{p-2}}{(p+2)G(\varphi)}dt=\int_0^x\frac{32b\varphi^{p-2}}{(p+2)(h_0^2-\varphi^2)}dt\ \leq \frac{32b}{(p+2)h_0}\int_0^{\infty}\frac{\varphi^{p-2}}{h_0-\varphi}dt,
\end{equation}
By Lemma \ref{lm:behave_varphi_zero mass}, there exists a sufficiently large $T>0$, independent of $b$, such that $\varphi(t)<\frac{h_0}{2}$ for all $t>T$. Therefore, from Lemma \ref{lm:behave_varphi_zero mass}, we have
\[
\int_T^{\infty}\frac{\varphi^{p-2}}{h_0-\varphi}dt<\frac{2}{h_0}\int_T^{\infty} t^{2-p}dt\lesssim 1,
\]
where the implicit constant is independent of $b$. Thus, to estimate $G_1$, it suffices to upper bound the quantity  
\begin{align*}
\int_0^T \frac{\varphi^{p-2}}{h_0-\varphi}dt.
\end{align*}
To achieve this goal, we first prove that there exists positive constants $A,B$ are independent of $b$ such that
\begin{equation}\label{eq:123}
h_0-\varphi(t)\geq Ab+Bt^2,\quad \forall t\in [0,T],
\end{equation}
Moreover,
\[
h_0-\varphi(t)=h_0-h_b+\varphi(0)-\varphi(t)\geq Ab+\varphi(0)-\varphi(t),
\] 
hence \eqref{eq:123} reduces to 
\begin{equation}\label{eq:234}
\varphi(t)\leq \varphi(0)-Bt^2,\quad\forall t\in [0,T].
\end{equation}
We have 
\[
G'(h_0) =\frac{-1}{8} h_0=-\lambda<0.
\]
Thus, there exists $\delta>0$ sufficiently small that is independent of $b$, such that
\[
G'(y)<-\frac{3\lambda}{4},\quad\forall y\in [h_0-\delta,h_0].
\]
Therefore, $\forall y\in [h_0-\delta,h_b]$, for sufficiently small $b$, we have
\[
G_b'(y)<G'(y)+|G'(y)-G_b'(y)|<\frac{-3\lambda}{4}+\frac{2b(p-2)h_0^{p-3}}{p+2}<-\frac{\lambda}{2}.
\] 
By continuity and decreasing of $\varphi$, there exists a unique $t_0>0$ such that
\[
h_b\geq \varphi(t)\geq h_0-\delta,\ \forall t\in [0,t_0],
\]
and
\[
h_0-\delta\geq \varphi(t)>0,\ \forall t\in [t_0,\infty),
\]
On $[0,t_0]$, since $G_b' \leq -\frac{\lambda}{2} < 0$, applying Lagrange theorem, we have:
\begin{equation*}
    G_b(\varphi)=G_b(\varphi)-G_b(h_b)=G_b'(\xi)(\varphi-h_b) \geq \frac{\lambda}{2}(h_b - \varphi)
\end{equation*}
Using the differential equation $\varphi' = -\varphi^2\sqrt{G_b(\varphi)}$, we obtain:
\begin{equation*}
    \frac{-\varphi'}{\sqrt{h_b - \varphi}} \geq (h_0-\delta)^2\sqrt{\frac{\lambda}{2}}. 
\end{equation*}
Integrating both sides from $0$ to $t$ ($t \le t_0$):
\begin{equation*}
    2\sqrt{h_b - \varphi(t)} \geq (h_0-\delta)^2\sqrt{\frac{\lambda}{2}} t \implies h_b - \varphi(t) \ge (h_0-\delta)^2\frac{\lambda}{8} t^2
\end{equation*}
\begin{equation*}
    \implies \varphi(t) - \varphi(0) \leq -\frac{\lambda (h_0-\delta)^2}{8} t^2 \quad \forall t \in [0, t_0]
\end{equation*}
For $t>t_0$, since $\varphi$ decreases on $\R^+$, we have $\varphi(t)<h_0 - \delta$. It follows that 
\begin{equation*}
    h_b - \varphi(t) \ge h_b - (h_0 - \delta) \ge \delta - O(b)
\end{equation*}
Thus, for $b$ sufficiently small, we have 
\begin{align*}
h_b - \varphi(t) &\ge \frac{\delta}{2} = \left(\frac{\delta}{2T^2}\right) T^2 \ge \left(\frac{\delta}{2T^2}\right) t^2
\end{align*}
It is equivalent to
\begin{equation*}
    \varphi(t) - h_b \le -\left(\frac{\delta}{2T^2}\right) t^2 \quad \forall t \in [t_0, T].
\end{equation*}
Letting $B=\min\left\{\frac{\lambda (k_0-\delta)^2}{8},\frac{\delta}{2T^2}\right\}$, \eqref{eq:234} is proved. \\
Using \eqref{eq:234} and the proof of Step 2 of Lemma \ref{lm2.5}, we have 
\begin{align*}
b\int_0^T \frac{\varphi^{p-2}}{h_0-\varphi}dt &\leq bh_0^{p-2}\int_0^T \frac{1}{Ab+Bt^2}dt\leq bh_0^{p-2}\frac{1}{\sqrt{ABb}}\arctan\left(T\sqrt{\frac{B}{Ab}}\right)\\
&\leq \frac{h_0^{p-2}\sqrt{b}}{\sqrt{AB}}\frac{\pi}{2}=O(\sqrt{b}).
\end{align*}
In conclusion, we have
\[
\sup_{x\in\R}|G_1(x)|=O(\sqrt{b}).
\]
\textbf{Step 4. Conclusion.}\\
By Step 2 and Step 3, we have
\[
\varphi^2(x)=\frac{4c}{1+c^2(x+O(\sqrt{b}))^2}.
\]
Comparing it to $\varphi_0^2$, we obtain, for $b$ sufficiently small,
\begin{align*}
|\varphi(x)-\varphi_0(x)|^2 &\leq |\varphi^2(x)-\varphi_0^2(x)|=\left|\frac{4c}{1+c^2(x+O(\sqrt{b}))^2}-\frac{4c}{1+c^2x^2}\right|\\
&\leq \frac{4c^3 O(\sqrt{b})(2|x|+O(\sqrt{b}))}{(1+c^2x^2)(1+c^2(x+O(\sqrt{b}))^2)}\\
&\leq \frac{4c^3 O(\sqrt{b})(2|x|+1)}{(1+c^2(x+O(\sqrt{b}))^2)(1+2c|x|)/2}\\
&\leq \frac{8c^3\left(1+\frac{1}{c}\right) O(\sqrt{b})}{(1+c^2(x+O(\sqrt{b}))^2)}\\
&=\frac{O(\sqrt{b})}{1+c^2(x+O(\sqrt{b}))^2}.
\end{align*} 
This yields the desired results \eqref{eq_prove_behave_solitons_zero_mass}.\\

\end{proof}

\subsection{Asymptotic behavior of soliton profile with respect to parameter $p$}

In this section, we study the changing of the soliton profile with respect to parameter $p$ by fixed the other parameters $\omega$, $c$ and $b$. Recall that when $p=4$, $b\geq 0$ the soliton profile $\varphi_0$ is of the form 
\begin{equation}
\varphi_0^2(x) = 
\begin{cases} 
\frac{8a}{\sqrt{c^2 + 4\gamma a} \cosh(2\sqrt{a} x)-c}  & \text{if } -2\sqrt{\omega} < c < 2\sqrt{\omega}, \\
\frac{4c}{(cx)^2+\gamma} & \text{if } c = 2\sqrt{\omega},
\end{cases}
\end{equation}
where $\gamma=1+\frac{16b}{3}$.\\
Let $b\geq 0$, $-2\sqrt{\omega} < c < 2\sqrt{\omega}$. We investigate the variation of the profile $\varphi$ as $p$ decreases to $4$. We have the following result.

\begin{lemma}\label{lm2.6}
Assume that 
\begin{equation}
\label{eq:assume k_0}
k_0>e^{\frac{1}{3}},
\end{equation}
where $k_0=\frac{1}{\gamma}\left( 2c+ 2\sqrt{c^2+4a\gamma}\right)$. The function $\varphi$ satisfies the following properties:
\begin{equation}\label{eq:need to prove lm 2.6}
\norm{\varphi-\varphi_0}_{L^{\infty}}=O(\sqrt[4]{p-4}),\quad \norm{\varphi-\varphi_0}_{L^2}=O(\sqrt[4]{p-4}).
\end{equation}
Moreover, 
\begin{equation}
\label{eq:need_prove2_lm2.6}
\norm{\partial_{\omega}\varphi-\partial_{\omega}\varphi_0}_{L^2}=O(\sqrt[4]{p-4}),
\end{equation}
and
\begin{equation}
\label{eq:need_prove3_lm2.6}
\norm{\partial_c\varphi-\partial_c\varphi_0}_{L^2}=O(\sqrt[4]{p-4}).
\end{equation}
\end{lemma}

\begin{proof}
We divide the proof in several steps. In what follows, we assume that $p$ is sufficiently close to $4$.\\
\textbf{Step 1: Formulating the profile $\varphi$.}\\ 
Let $y=\varphi^2$. As $\varphi$ is an even function, so is $y$. Multiply both sides of \eqref{eqofvarphi2} by $\varphi$, we have
\begin{equation}\label{eq_of_y_p}
y'=-y\sqrt{4a+cy-\frac{1}{4}y^2-\frac{8b}{p+2}y^{p/2}}.
\end{equation}
For convenience, let 
\[
K(y)=4a+cy-\frac{\gamma}{4} y^2,\quad K_p(y)=K(y)-\frac{8b}{p+2}y^{p/2}+\frac{4b}{3}y^2=4a+cy-\frac{1}{4}y^2-\frac{8b}{p+2}y^{p/2}.
\]
We have $$K(y)=\frac{\gamma}{4}(y_+ - y)(y-y_{-}),$$ with $y_{\pm}=\frac{1}{\gamma}\left( 2c\pm 2\sqrt{c^2+4a\gamma}\right)$. It is easy to verify that $y_+>0>y_{-}$. \\

Let $k_p$ be the maximal value of $y$ on $\R$ i.e $k_p=y(0)$. When $p=4$, let $y_0=\varphi_0^2$ and the maximal value of $y_0$ on $\R$
\[
k_0=y_0(0)=y_+,
\]  
We have
\[
K_p(k_0)=K(k_0)-\frac{8b}{p+2}k_0^{p/2}+\frac{4b}{3}k_0^2=-\frac{8b}{p+2}k_0^{p/2}+\frac{4b}{3}k_0^2.
\]
The function $f(p)=\frac{k_0^{p/2}}{p+2}$ has a derivative at $p=4$ equal to
\[
f'(4)=\frac{f(4)}{2}\left(\ln(k_0)-\frac{1}{3}\right)>0,
\]
by the assumption \eqref{eq:assume k_0}. Hence, $K_p(k_0)<0$ provided that $p$ is sufficiently close to $4$. Moreover, $K_p(t)>0$ for all $t\in (0,k_p)$. This implies that $k_0>k_p\geq y(x)$ for all $x\in\R$. Consequently, $K(y)>\frac{\gamma}{4}(k_0-y)(-y_{-})>0$ for all $x\in\R$.\\

Since $k_p$ is the first positive root of $K_p$, we have, for $p$ sufficiently close to $4$,
\begin{align*}
0&=K_p(k_p)=K(k_p)-\frac{8b}{p+2}k_p^{p/2}+\frac{4b}{3}k_p^2\\
&= \frac{\gamma}{4}(y_+ - k_p)(k_p-y_{-})-\frac{8b}{p+2}k_p^{p/2}+\frac{4b}{3}k_p^2.
\end{align*}
This yields
\begin{align*}
\frac{\gamma}{4}(k_0-k_p)(-y_{-})&<\frac{\gamma}{4}(y_+ - k_p)(k_p-y_{-})=\frac{8b}{p+2} k_p^2\left(k_p^{(p-4)/2}-\frac{p+2}{6}\right)\\
&<\frac{8b}{6} k_0^2\left(k_0^{(p-4)/2}-\frac{p+2}{6}\right)=\frac{4b k_0^2}{3} O(p-4).
\end{align*}
This implies that
\begin{equation}
\label{eq:estimate k_0-k_p}
0<k_0-k_p<\frac{4}{\gamma|y_{-}|}\frac{4bk_0^2}{3}O(p-4)=O(p-4),
\end{equation}

From \eqref{eq_of_y_p}, we have
\begin{equation}\label{eq_of_y_2_p}
y'=-y\sqrt{K_p(y)}=-y\sqrt{K(y)}+\frac{8by^3}{p+2}\frac{y^{(p-4)/2}-\frac{p+2}{6}}{\sqrt{K(y)}+\sqrt{K_p(y)}}.
\end{equation}
Thus, dividing both sides by $y\sqrt{K(y)}$ and integrating from $0$ to $x$, we obtain
\begin{equation}\label{eq_varphi_o(p)}
\int_0^x \frac{y'}{y\sqrt{K(y)}} dt= \int_0^x -1+\frac{8by^2}{p+2}\frac{y^{(p-4)/2}-\frac{p+2}{6}}{\sqrt{K(y)}(\sqrt{K(y)}+\sqrt{K_p(y)})} dt,\quad \forall x>0.
\end{equation}
By changing variable $y=\frac{y_+}{1+z}$, we have
\begin{align}
\int \frac{y'}{y\sqrt{K(y)}} dt &=\frac{2}{\sqrt{\gamma}}\int \frac{dy}{y\sqrt{(y_+ - y)(y-y_{-})}}\nonumber\\
&=-\frac{2}{\sqrt{\gamma}}\int \frac{dz}{\sqrt{y_+ z (y_+ - y_{-}- y_{-} z)}}\nonumber\\
&=-\frac{2}{\sqrt{\gamma}}\int \frac{dz}{\sqrt{\frac{16a}{\gamma}((z+C_0)^2-C_0^2)}},\quad (\text{ with } 32aC_0=\gamma y_+(y_+ -y_{-})),\nonumber\\
&=\frac{-1}{2\sqrt{a}}\int \frac{d\hat{z}}{\sqrt{\hat{z}^2-C_0^2}},\quad (\text{ with } \hat{z}=z+C_0),\nonumber\\
&=\frac{-1}{2\sqrt{a}}\int \frac{dw}{\sqrt{w^2-1}},\quad (\text{ with } \hat{z}=C_0w),\nonumber\\
&=\frac{-1}{2\sqrt{a}}\text{arcosh}(w).\label{eq_arcoshw_p}
\end{align}
Note that from inequality $y_+> y(x)> y_{-}$, it follows that $$w(x)=\frac{y_+}{C_0 y(x)}+\frac{C_0-1}{C_0}>1,\quad \text{ for all } x\in\R.$$ 
Therefore, the function $\text{arcosh}(w)$ is well-defined on $\R$. \\
Combining \eqref{eq_arcoshw_p} with \eqref{eq_varphi_o(p)} and $y=\varphi^2$, we have 
\[
\frac{-1}{2\sqrt{a}}\left(\text{arcosh}\left(\frac{y_+}{C_0\varphi^2(x)}+\frac{C_0-1}{C_0}\right)-\text{arcosh}(w(0))\right)=-x+K_1(x),
\]
where,
\[
K_1(x)=\int_0^x \frac{8by^2}{p+2}\frac{y^{(p-4)/2}-\frac{p+2}{6}}{\sqrt{K(y)}(\sqrt{K(y)}+\sqrt{K_p(y)})}  dt, \quad \forall x>0.
\]
This implies that
\[
\varphi^2(x)=\frac{1}{\frac{C_0}{y_+}\cosh(2\sqrt{a}x+\text{arcosh}(w(0))-2\sqrt{a}K_1(x))-\frac{C_0-1}{y_+}}.
\]
Substituting $C_0=\frac{\gamma y_+(y_+ - y_{-})}{32a}$, $y_{\pm}=\frac{1}{\gamma}\left(2\pm 2\sqrt{c^2+4a\gamma}\right)$ into the above expression, we get that
\begin{equation}\label{eq_varphi_formula_p}
\varphi^2(x)=\frac{8a}{\sqrt{c^2+4a\gamma} \cosh(2\sqrt{a}x+\text{arcosh}(w(0))-2\sqrt{a}K_1(x))-c}.
\end{equation}

\textbf{Step 2: Estimate $w(0)$}\\

We have $w(0)>1$ and
\[
w(0)=\frac{y_+}{C_0y(0)}+\frac{C_0-1}{C_0}=\frac{k_0}{C_0 k_p}+\frac{C_0-1}{C_0}=1+O(p-4).
\]
It follows that, for $b$ small enough,
\begin{align*}
\text{arcosh}(w(0))&=\ln (w(0)+\sqrt{w(0)^2-1})=\ln(1+O(p-4)+\sqrt{(1+O(p-4))^2-1})\\
&=\ln(1+O(p-4))=O(p-4).
\end{align*}

\textbf{Step 3:Estimate $K_1$.} By oddness of $K_1$, we just need to estimate $K_1$ on $(0,\infty)$.\\
Noting that $K(y)\approx (k_0 -y)(y-y_{-})\approx k_0-y$, with the implicit constant is independent of $p$, we have
\begin{equation}
\label{eq:approx K_1}
|K_1(x)|\lesssim \int_0^x \frac{8by^2}{p+2}\frac{\left|y^{(p-4)/2}-\frac{p+2}{6}\right|}{K(y)}dt\approx \int_0^x \frac{y^2\left|y^{(p-4)/2}-\frac{p+2}{6}\right|}{k_0-y}dt,
\end{equation}
where the implicit constant is independent of $p$. It follows that
\[
\sup_{x\in\R}|K_1(x)| \lesssim \int_0^{\infty}\frac{y^2}{k_0-y}\left|y^{(p-4)/2}-\frac{p+2}{6}\right| dt.
\]
By Lemma \ref{Taylor expansion}, applying the Taylor expansion at zero to the function $y(x) = y^x$ with respect to the variable $x = \frac{p-4}{2}$, we have, for some $\xi \in (0, (p-4)/2)$,
\begin{align*}
|y^{(p-4)/2}-1|&=\left|\ln(y)\frac{p-4}{2} + y^{\xi} \left[(\xi\ln y)^2+\ln y\right]\frac{(p-4)^2}{8}\right|\\
&\lesssim |\ln(y)|(p-4)+|\ln (y)|^2(p-4)^2\lesssim \sqrt{p-4},
\end{align*}
provided that $y<1$ and $-\ln(y)= |\ln(y)|<\frac{1}{\sqrt{p-4}}$ or $e^{-\frac{1}{\sqrt{p-4}}}<y<1$. Since $y(0)=k_p\approx k_0>1$, there exists $T_1<T_2$ such that $y(T_1)=1$ and $y(T_2)=e^{-\frac{1}{\sqrt{p-4}}}$. The integral from $T_1$ to $T_2$ is admissible i.e
\[
\int_{T_1}^{T_2}\frac{y^2}{k_0-y}\left|y^{(p-4)/2}-\frac{p+2}{6}\right| dt\lesssim \sqrt{q-4} \int_{T_1}^{T_2}\frac{y^2}{k_0-1}dt\lesssim \sqrt{q-4} \norm{y}^2_{L^2}=O(\sqrt{q-4}).
\] 
Moreover, $\sqrt{p-4}> e^{-\frac{1}{\sqrt{p-4}}}>y(t)>0$ for all $t\in (T_2,\infty)$. It follows that
\begin{align*}
\int_{T_2}^{\infty}\frac{y^2}{k_0-y}\left|y^{(p-4)/2}-\frac{p+2}{6}\right| dt &\lesssim \sqrt{p-4}\int_{T_2}^{\infty} y(t)dt =O(\sqrt{p-4}).
\end{align*} 
Thus, it remains to estimate the integral over $(0,T_1)$. Noting that $T_1$ is independent of $p$ and $k_0>k_p>y(t)>1$ on $(0,T_1)$, we have
\begin{align}
\int_{0}^{T_1}\frac{y^2}{k_0-y}\left|y^{(p-4)/2}-\frac{p+2}{6}\right| dt &\lesssim \left(k_0^{(p-4)/2}-1+\frac{p-4}{6} \right)\int_0^{T_1} \frac{1}{k_0-y}dt\nonumber\\
&\leq C(p-4)\int_0^{T_1}\frac{1}{k_0-y}dt.\label{eq:345}
\end{align} 
Similarly to the proof of Lemma \ref{lm2.5}, we have
\[
\int_0^{T_1}\frac{1}{k_0-y}dt \lesssim \frac{1}{\sqrt{p-4}}.
\]
Thus, by \eqref{eq:345}, we have
\[
\int_{0}^{T_1}\frac{y^2}{k_0-y}\left|y^{(p-4)/2}-\frac{p+2}{6}\right| dt=O(\sqrt{p-4}).
\] 
In conclusion,
\[
\sup_{x\in\R}|K_1(x)|=O(\sqrt{p-4}).
\]
\textbf{Step 4: Conclusion.}\\
From Step 2, Step 3 and \eqref{eq_varphi_formula_p}, we have
\[
\varphi^2(x)=\frac{8a}{\sqrt{c^2+4a\gamma} \cosh(2\sqrt{a}x+O(\sqrt{p-4}))-c}.
\]
Similarly to the proof of Lemma \ref{lm2.5}, by letting $z=2\sqrt{a}x$, we have
\begin{align*}
|\varphi(x)-\varphi_0(x)|^2 &\le |\varphi^2(x)-\varphi_0^2(x)|\leq\frac{O(\sqrt{p-4})}{(|c|+\varepsilon) \cosh(z)-c}.
\end{align*}
This easily yields \eqref{eq:need to prove lm 2.6}, while \eqref{eq:need_prove2_lm2.6} and \eqref{eq:need_prove3_lm2.6} can be established similarly to \eqref{eq_need_prove2} and \eqref{eq_need_prove3}, respectively. We only need to verify that the quantity $L_0-L_p$ is admissible. From \eqref{eq:need to prove lm 2.6}, it is easy to verify that
\[
|L_0-L_p|=O(\sqrt[4]{p-4})+b(p+1)\varphi_0^4(\varphi_0^{p-4}-1).
\]
To estimate the last term, we utilize the similar approach to bound the term $K_1$ in Step 3. Let $T_1<T_2$ be two positive constants such that $\varphi_0(T_1)=1$, $\varphi_0(T_2)=e^{-\frac{1}{\sqrt{p-4}}}$. On $(T_1,T_2)$, $e^{-\frac{1}{\sqrt{p-4}}}<\varphi_0<1$. Thus, there exists $\xi\in(0,p-4)$ such that
\begin{align*}
|\varphi_0^{p-4}-1|&=\left|\ln(\varphi_0)(p-4)+\varphi^{\xi}\left[(\xi\ln\varphi_0)^2+\ln\varphi_0\right]\frac{(p-4)^2}{2}\right|\lesssim \sqrt{p-4}.
\end{align*}
When $x>T_2$, we have $0<\varphi_0<e^{-\frac{1}{\sqrt{p-4}}}$. It follows that
\[
|\varphi_0^4||\varphi_0^4-1|\lesssim \varphi_0^4 <(p-4)^2.
\]
On $(0,T_1)$, we have $1<\varphi_0<\sqrt{k_0}$. Thus, 
\[
|\varphi^{p-4}-1|\lesssim |k_0^{\frac{(p-4)}{2}}-1|\lesssim p-4. 
\]
In conclusion, $L_0-L_p$ is admissible i.e
\[
\sup_{x\in\R}|(L_0-L_p)(x)|=O(\sqrt[4]{p-4}).
\]
This completes the proof of Lemma \ref{lm2.6}.\\

\end{proof}

\subsection{Stability of solitons}

In this section, we present several stability and instability results for the solitons introduced in the previous section. The proofs of these results follow similar arguments to those in \cite{CoWu18,LiSiSu13,CoOh06}.

\begin{theorem}\label{sta_from_cowu}
Assume that there exists $\mu\in\R$ such that for any $\varepsilon\in H^1(\R)$ verifying the orthogonality conditions
\begin{equation}\label{eq_ortho_condition}
(\varepsilon,i\phi_{\omega,c})_2=(\varepsilon,\partial_x\phi_{\omega,c})_2=(\varepsilon,\phi_{\omega,c}+i\mu\partial_x\phi_{\omega,c})_2=0,
\end{equation}
we have
\[
\left<S_{\omega,c}''(\phi_{\omega,c})\varepsilon,\varepsilon\right>_2\gtrsim  \norm{\varepsilon}_{H^1}^2,
\]
where
\[
S_{\omega,c}''(\phi_{\omega,c})u:=(E''(\phi_{\omega,c})+\omega M''(\phi_{\omega,c})+cP''(\phi_{\omega,c}))u,\quad\forall u\in H^2(\R).
\]
Then, the associated solitary wave $u_{\omega,c}(t,x)=e^{it\omega}\phi_{\omega,c}(x-ct)$ of \eqref{eq1} is orbitally stable. 
\end{theorem}

\begin{proof}
The proof closely follows the argument in \cite{We85} (see also \cite[Remark 2.3]{CoWu18}).
\end{proof}

\begin{remark}\label{rm:2.9}
Similarly to \cite{CoWu18}, we have Ker$(S_{\omega,c}''(\phi_{\omega,c}))=\{i\phi_{\omega,c},\partial_x\phi_{\omega,c}\}$ and the negative part of $S_{\omega,c}''(\phi_{\omega,c})$ has only one dimension. Moreover, from \cite[Step 4, Proof of Proposition 2.1]{CoWu18}, if there exists $\mu\in\R$ such that for $\psi=\partial_{\omega}\phi_{\omega,c}+\mu\partial_c\phi_{\omega,c}$, we have
\[
\left<S_{\omega,c}''(\phi_{\omega,c})\psi,\psi\right> <0,
\]
then the conclusion of Theorem \ref{sta_from_cowu} hold. \\
From \cite{CoWu18}, we have
\[
\left<S_{\omega,c}''(\phi_{\omega,c})\psi,\psi\right> =-\partial_{\omega}M(\phi_{\omega,c})-\mu\partial_cM(\phi_{\omega,c})-\mu\partial_{\omega}P(\phi_{\omega,c})-\mu^2\partial_cP(\phi_{\omega,c}).
\]
Therefore, if $\partial_{\omega}M(\phi_{\omega,c})>0$ then $\left<S_{\omega,c}''(\phi_{\omega,c})\psi,\psi\right>  < 0$ for $\mu>0$ small enough and if $\partial_cP(\phi_{\omega,c})>0$ then $\left<S_{\omega,c}''(\phi_{\omega,c})\psi,\psi\right>  < 0$ for $\mu>0$ large enough. In these cases, the coercive property in Theorem \ref{sta_from_cowu} holds and the solitary wave is orbitally stable. 
\end{remark}

A pair $(\omega,c)\in\R^2$ is called an admissible pair if $\omega>\frac{c^2}{4}$. Let $\Omega$ be the set of all admissible pairs. Analogously to \cite{LiSiSu13}, we have the following results. 
\begin{theorem}\label{thm*}
For all $b\ge 0$ and admissible $(\omega,c)$, the space $H^1$ can be decomposed as the direct sum:
\[
H^1=N+Z+P,
\]
where the three subspaces intersect trivially and:
\begin{itemize}
\item[i)] $N$ is a one dimensional subspace such that for $u\in N$, $u\neq 0$,
\[
\left<S_{\omega,c}''(\phi_{\omega,c})u,u\right><0.
\]
\item[ii)] $Z$ is the two dimensional kernel of $S_{\omega,c}''(\phi_{\omega,c})$.
\item[iii)] $P$ is a subspace such that, for $p\in P$, 
\[
\left<S_{\omega,c}''(\phi_{\omega,c})p,p\right> \ge\delta\norm{p}_{H^1}^2,
\]
where the constant $\delta>0$ is independent of $p$.
\end{itemize}
\end{theorem}

As a consequence of Theorem \ref{thm*}, we have the following result.
\begin{corollary}\label{coro_negative_part}
For all $b\ge 0$ and admissible $(\omega,c)$. 
\[
n(S_{\omega,c}''(\phi_{\omega,c}))=1.
\]
\end{corollary}

\begin{theorem}(Grillakis, Shatah, Strauss \cite{GrShSt87,GrShSt90})\label{Gr_Sh_St}
\[
p(d'')\le n(S_{\omega,c}''(\phi_{\omega,c})),
\]
where $d(\omega,c)=S_{\omega,c}(\phi_{\omega,c})$ and $p(d'')$ is number of positive eigenvalue of the matrix $d''$ defined by
\[
d''(\omega,c)=\left(\begin{matrix}
\partial_{\omega}M(\phi_{\omega,c}) & \partial_cM(\phi_{\omega,c})\\
\partial_{\omega}P(\phi_{\omega,c}) & \partial_cP(\phi_{\omega,c})
\end{matrix}\right).
\]
Furthermore, if $d$ is non-degenerate at $(\omega,c)$ then the following result hold:
\begin{itemize}
\item[i)] If $p(d'')=n(S_{\omega,c}''(\phi_{\omega,c}))$ then the solitary wave is orbitally stable;
\item[ii)] If $n(S_{\omega,c}''(\phi_{\omega,c}))-p(d'')$ is odd then the solitary wave is orbitally unstable.
\end{itemize}
\end{theorem} 

Combining Corollary \ref{coro_negative_part} and Theorem \ref{Gr_Sh_St}, we have the following result.
\begin{theorem}\label{sta_by_Gr_Sh_St}
The solitary wave $u_{\omega,c}(t,x)=e^{i\omega t}\phi_{\omega,c}(x-ct)$ of \eqref{eq1} is orbitally stable if $p(d'')=1$ and orbitally unstable if $p(d'')=0$. 
\end{theorem}

Inspiring by the results in \cite[Theorem 3]{CoOh06}, we obtain following result.
\begin{theorem}
Let $(\omega,c)$ be an admissible pair. Assume that there exists $\xi\in\R^2$ such that
\[
\left<d'(\omega,c),\xi\right>\neq 0,\quad\left<d''(w,c)\xi,\xi\right>>0,
\]
where $\left<,\right>$ denotes the inner product in $\R^2$ and $d(\omega,c)$ is defined as in Theorem \ref{Gr_Sh_St}. Then the solitary wave $u_{\omega,c}(t,x)=e^{i\omega t}\phi_{\omega,c}(x-ct)$ is orbitally stable.
\end{theorem}

\begin{proof}
The proof follows a similar argument to that of \cite[Theorem 3]{CoOh06}, with the operators $L_{\omega,c}$ and $N$ replaced by $\tilde{L}_{\omega,c}$ and $\tilde{N}$, respectively, where
\begin{align*}
\tilde{N}(u)&=N(u)+\frac{2b(p-2)}{p+2}\int_{\R}|u|^{p+2}dx,\\
\tilde{L}_{\omega,c}(u)&=L_{\omega,c}(u)+\frac{b(p-2)}{p+2}\int_{\R}|u|^{p+2}dx.
\end{align*}
We omit the details here.
\end{proof}

As a consequence of the above theorem, we have the following result.
\begin{corollary}\label{coro**}
Let $(\omega,c)$ be an admissible pair. Assume that det$[d''(w,c)]<0$ or $\partial_{\omega}^2d(\omega,c)=\partial_{\omega}M(\phi_{\omega,c})>0$. Then the solitary wave $u_{\omega,c}$ is orbitally stable. 
\end{corollary}

Combining the above result with Lemma \ref{lm2.5}, we obtain the proof of Theorem \ref{thm:main_1}. Similarly, Theorem \ref{thm:main_2} follows from Lemma \ref{lm2.6} and Corollary \ref{coro**}.
%\begin{theorem}\label{thm:main_1}
%Let $p\geq 4$ and $(\omega,c)$ be an admissible pair. Then, there exists $b_0=b_0(\omega,c)$ sufficiently small such that for all $b\in (0,b_0)$, the solitary wave $u_{\omega,c}(t,x)=e^{i\omega t}\phi_{\omega,c}(x-ct)$ of \eqref{eq1} is orbitally stable.
%\end{theorem}

\begin{proof}[Proof of Theorem \ref{thm:main_1}]
%If $c<0$ then $\int_{\R}\partial_{\omega}\varphi_0 \varphi_0 dx=\partial_{\omega}M(\varphi_0)>0$ (see e.g \cite[Remark 3]{CoOh06}). From Lemma \ref{lm2.5}, this implies that 
%\[
%\partial_{\omega}M(\phi_{\omega,c})=\partial_{\omega}M(\varphi_{\omega,c})>0,
%\]
%provided that $b>0$ sufficiently small. Thus, from Corollary \ref{coro**}, $u_{\omega,c}$ is orbitally stable.\\
%If $c\ge 0$ then
We have
\[
 \text{det}(d''(\omega,c,0)):=\partial_{\omega}M(\phi_0)\partial_cP(\phi_0)-\partial_{\omega}P(\phi_0)\partial_cM(\phi_0)=\frac{-1}{\omega}<0,
\] 
where $\phi_0(x)=\varphi_0(x)\exp\left(\frac{c}{2}ix-\frac{i}{4}\int_{-\infty}^x |\varphi_0(y)|^2dy\right)$. For the proof, we refer the reader to \cite{CoOh06}. \\
Moreover,
\[
P(\phi_{\omega,c})=\frac{-c}{4}\int_{\R}\varphi_{\omega,c}^2dx+\frac{1}{8}\int_{\R}\varphi_{\omega,c}^4 dx.
\] 
Thus,
\begin{align*}
\partial_{\omega}P(\phi_{\omega,c})&=\frac{-c}{2}\int_{\R}\partial_{\omega}\varphi_{\omega,c} \varphi_{\omega,c} dx+\frac{1}{2}\int_{\R}\partial_{\omega}\varphi_{\omega,c}\varphi_{\omega,c}^3 dx,\\
\partial_c P(\phi_{\omega,c})&=\frac{-c}{2}\int_{\R}\partial_c\varphi_{\omega,c} \varphi_{\omega,c} dx+\frac{1}{2}\int_{\R}\partial_c\varphi_{\omega,c}\varphi_{\omega,c}^3 dx-\frac{1}{4}\int_{\R}\varphi_{\omega,c}^2dx.
\end{align*}
Hence, by Lemma \ref{lm2.5}, we have
\[
\text{det}(d''(\omega,c)):=\partial_{\omega}M(\phi_{\omega,c})\partial_cP(\phi_{\omega,c})-\partial_{\omega}P(\phi_{\omega,c})\partial_cM(\phi_{\omega,c})=\text{det}(d''(\omega,c,0))+O(\sqrt[4]{b})<0,
\]
provided that $b$ is sufficiently small. This implies that the solitary wave $u_{\omega,c}$ of \eqref{eq1} is orbitally stable by Corollary \ref{coro**}.

\end{proof}

%\begin{theorem}\label{thm:main_2}
%Let $b\geq 0$ and $(\omega,c)$ be an admissible pair such that $-2\sqrt{\omega}<c<2s^*\sqrt{\omega}$ and $k_0=\frac{1}{\gamma}\left(2c+2\sqrt{c^2+4a\gamma}\right)>e^{\frac{1}{3}}$. Then, there exists $p_0=p_0(\omega,c)$ sufficiently close to $4$ such that for all $p \in [4,p_0)$, the solitary wave $u_{\omega,c}(t,x)=e^{i\omega t}\phi_{\omega,c}(x-ct)$ of \eqref{eq1} is orbitally stable.
%\end{theorem}

\section{Scattering theory}\label{sec:scattering theory}

%check more
%In the case $b=0$, the authors \cite{BaWuXu20} proved that the solution is global and scatters for small initial data in $H^s(\R)$, $\frac{1}{2}\leq s\leq 1$ when $\sigma\ge 2$, while the opposite result fails to hold when $0<\sigma<2$. When $b\neq 0$ and $p\geq 4$, the same result holds and the scattering solutions are difficult to obtain. More precisely, we have the following result.
%\begin{lemma}
%\label{lm:scattering_theory_1}
%Let $(\omega, c)$ be such that $4\omega>c^2$. Let $\varphi_{\omega,c}$ be the positive, even solution of \eqref{eqofvarphi} and $\phi_{\omega,c}$ be defined in \eqref{eq:varphi_change to_phi}, then
%\[
%\|\phi_{\omega,c}\|_{H^1(\mathbb{R})} \to 0,\quad  \text{when } c \to -2\sqrt{\omega}.
%\]
%\end{lemma}
In this section, we sketch the proof of Lemmas \ref{lm:scattering_theory_1}, \ref{thm:modified scattering operator} and \ref{thm:non exists non trivial scattering solutions}.

\begin{proof}[Proof of Lemma \ref{lm:scattering_theory_1}]
Throughout the proof, we fix $b,\omega>0$ and let $c\rightarrow -2\sqrt{\omega}$. Note that $a\rightarrow 0$ when $c\rightarrow -2\sqrt{\omega}$. It suffices to prove $\|\varphi_{\omega,c}\|_{H^1}$ goes to zero. For convenience, we omit the subscript $\omega,c$. Multiplying both sides of \eqref{eqofvarphi} by $\varphi$ and integrating over $\R$ yields:
\[
\norm{\varphi_x}^2_{L^2}+\left(\omega-\frac{c^2}{4}\right)\norm{\varphi}^2_{L^2}+\frac{c}{2}\norm{\varphi}^4_{L^4}-\frac{3}{16}\norm{\varphi}^6_{L^6}-b\norm{\varphi}^{p+2}_{L^{p+2}}=0.
\]
Thus, we only need prove that
\[
\norm{\varphi}_{L^2}\rightarrow 0,\quad \norm{\varphi}_{L^{\infty}}\rightarrow 0.
\]
For this purpose, we use the formula of $\varphi$ \eqref{eqofvarphi}. We have
\begin{align*}
w(0)&=\frac{k_0}{C_0 k_b}+\frac{C_0-1}{C_0},
\end{align*}
where,
\begin{align*}
C_0&=\frac{y_+(y_+ - y_{-})}{32a}=\frac{\sqrt{c^2+4a}}{\sqrt{c^2+4a}-c}\rightarrow \frac{1}{4}.
\end{align*}
It follows that
\[
w(0)\approx \frac{4k_0}{k_b}-3=1+4\left(\frac{k_0-k_b}{k_b}\right)=1+O(k_0^{p/2})\rightarrow 1.
\]
Thus, 
\begin{align*}
\norm{\varphi}^2_{L^2}&=\int_{\R}\varphi^2(x)dx=2\int_0^{\infty}\varphi^2(x)dx\leq \frac{16a}{|c|}\int_0^{\infty}\frac{1}{\cosh(2\sqrt{a}x)+1}dx=\frac{8\sqrt{a}}{|c|}\int_0^{\infty}\frac{1}{\cosh(x)+1}dx, 
\end{align*}
which tends to zero when $c\rightarrow -2\sqrt{\omega}$. \\
Moreover,
\begin{align*}
\norm{\varphi}_{L^{\infty}}&=\varphi(0)=k_b<k_0=y_+=\frac{8a}{\sqrt{c^2+4a}-c},
\end{align*}
which goes to zero, and consequently, so does $\norm{\varphi}_{L^{\infty}}$. The proof is complete.
\end{proof}

\begin{proof}[Proof of Theorem \ref{thm:modified scattering operator}]
The proof is similar to \cite{HaOz94}. We only sketch this here. Let $\psi=\mathcal{G}u$. For each solution $u$ to \eqref{eq1}, $\psi$ satisfies
\[
i\psi_t+\psi_{xx}+i(|\psi|^2\psi)_x+b|\psi|^p\psi=0.
\]
For given $\phi^+$, we let
\[
v_+(t) = \frac{1}{\sqrt{2it}} \exp \left( \frac{ix^2}{4t} + iS^{+}(t, x) \right) \hat{\phi}_{+} \left( \frac{x}{2t} \right),
\]
Noting that the perturbation $|v_+|^pv_+$ is small i.e
\[
\norm{|v_+|^pv_+}_{2,0}\leq C\norm{v_+}_{2,0}\norm{v_+}_{L^{\infty}}^p+\norm{v_+}^{p-1}_{L^{\infty}}\norm{\partial_xv_+}_{L^{\infty}}\norm{\partial_xv_+}_{L^2}\leq C |t|^{-\frac{p}{2}},
\]
hence, from \cite{HaOz94}, $v_+$ satisfies
\[
i\partial_tv_+ +\partial_x^2v_+  + i(|v_+|^2 v_+)_x+b|v_+|^pv_+ =F_+(t),
\]
\[
\norm{F_+}_{2,0}\leq C|t|^{-2}(\ln|t|)^2\quad \text{ for any } t\geq e.
\]
Consider the Gauge transformation:
\[
w^{(1)}(t,x)=\exp\left(i\int_{-\infty}^x|\psi(t,y)|^2dy\right)\psi(t,x),
\]
\[
w^{(2)}(t,x)=\exp\left(i\int_{-\infty}^x|\psi(t,y)|^2dy\right)\left(\psi_x+\frac{i}{2}|\psi|^2\psi\right),
\]
\[
w^{(1)}_+(t,x)=\exp\left(i\int_{-\infty}^x|v_+(t,y)|^2dy\right)v_+(t,x),
\]
\[
w^{(2)}_+(t,x)=\exp\left(i\int_{-\infty}^x|v_+(t,y)|^2dy\right)\left(\partial_xv_{+}+\frac{i}{2}|v_+|^2v_+\right).
\]
The above functions satisfy the following systems
\begin{equation}\label{eq:syst1}
\begin{cases}
Lw^{(1)}=i|w^{(1)}|^2\overline{w^{(2)}}-b|w^{(1)}|^pw^{(1)}:=P(w^{(1)},w^{(2)}),\\
Lw^{(2)}=-i(w^{(2)})^2\overline{w^{(1)}}-b\left(\frac{p}{2}+1\right)|w^{(1)}|^pw^{(2)}-\frac{bp}{2}|w^{(1)}|^{p-2}(w^{(1)})^2\overline{w^{(2)}}:=Q(w^{(1)},w^{(2)}).
\end{cases}
\end{equation}
and
\begin{equation}\label{eq:syst2}
\begin{cases}
Lw^{(1)}_+ =P(w^{(1)}_+,w^{(2)}_+)+F_1^+,\\
Lw^{(2)}_+=Q(w^{(1)}_+,w^{(2)}_+)+F_2^+,
\end{cases}
\end{equation}
where $F_1^+$ and $F_2^+$ satisfy that
\[
F_1^+=F_+\exp\left(i\int_{-\infty}^x|v_+(t,y)|^2dy\right)-2w^1_+\int_{-\infty}^x\Im(F_+\overline{v_+})dy,
\]
\[
F_2^+=-2w^{(2)}_+\int_{-\infty}^x \Im(F_+ \overline{v}_+)dy+\exp\left(i\int_{-\infty}^x|v_+(t,y)|^2dy\right)\left(\partial_xF_+ + i|v_+|^2F_+ -\frac{i}{2}v^2_+ \overline{F}_+\right)
\]
Therefore, it is easy to verify that
\[
\norm{F_1^+}_{1,0}+\norm{F_2^+}_{1,0}\leq Ct^{-2}(\ln t)^2,\quad \forall t\geq e.
\]
Let $\eta=\begin{pmatrix} w^{(1)} \\ w^{(2)} \end{pmatrix}-\begin{pmatrix} w^{(1)}_+ \\ w^{(2)}_+ \end{pmatrix}$, $\mathcal{W}=\begin{pmatrix} w^{(1)}_+ \\ w^{(2)}_+ \end{pmatrix}$, $\mathcal{K}=\begin{pmatrix} P\\ Q\end{pmatrix}$, $\mathcal{F}=\begin{pmatrix} -F_1^+\\ -F_2^+\end{pmatrix} $. From the systems \eqref{eq:syst1} and \eqref{eq:syst2}, we have
\begin{equation}\label{eq:eta_0}
L\eta = \mathcal{K}(\eta+\mathcal{W})-\mathcal{K}(\mathcal{W})+\mathcal{K}.
\end{equation}
We find a solution $\eta$ of the following asymptotic problem
\begin{equation}
\label{eq:eta}
\eta= i\int_t^{\infty} U(t-\tau)(\mathcal{K}(\eta+\mathcal{W})-\mathcal{K}(\mathcal{W})+\mathcal{F})(\tau)d\tau:=\Phi(\eta).
\end{equation}
For convenience, for any $f=(f_1,f_2)$, we denote $|f|=|f_1|+|f_2|$. Let $\alpha\in (1/2,1)$.\\
Let
$$
\begin{aligned}
X_{T,M} = \Big\{ \eta \in C([T,\infty),H^{1,0}\times H^{1,0}) &\cap L^4((T,\infty),W^{1,\infty}\times W^{1,\infty}) \\
&\Bigm\vert \sup_{t\geq T}t^{\alpha}\left(\norm{\eta}_{L^{\infty}((t,\infty),H^{1,0}\times H^{1,0})}+\norm{\eta}_{L^4((t,\infty),W^{1,\infty}\times W^{1,\infty})}\right) \leq M \Big\},
\end{aligned}
$$
which is established the metric:
\[
d(\eta,\tilde{\eta})=\sup_{t\geq T}t^{\alpha}\left(\norm{\eta-\tilde{\eta}}_{L^{\infty}((t,\infty),H^{1,0}\times H^{1,0})}+\norm{\eta-\tilde{\eta}}_{L^4((t,\infty),W^{1,\infty}\times W^{1,\infty})}\right).
\]
For any $\eta_1,\eta_2\in X_{T,M}$, we have
\[
\left|\mathcal{K}(\eta_1)-\mathcal{K}(\eta_2)\right|\lesssim |\eta_1-\eta_2|(|(\eta_1,\eta_2)|^2+|(\eta_1,\eta_2)|^p),
\]
and
\begin{align*}
\left|\partial_x\mathcal{K}(\eta_1)-\partial_x\mathcal{K}(\eta_2)\right|&\lesssim |\partial_x\eta_1-\partial_x\eta_2|(|(\eta_1,\eta_2)|^2+|(\eta_1,\eta_2)|^p)\\
& \quad +|\eta_1-\eta_2|(|(\eta_1,\eta_2)|+|(\eta_1,\eta_2)|^{p-1})||(\partial_x\eta_1,\partial_x\eta_2)|.
\end{align*}
Therefore, by Strichartz's estimates and Hölder inequality, we have
\begin{align*}
&\norm{\Phi(\eta_1)-\Phi(\eta_2)}_{L^{\infty}((t,\infty),H^{1,0}\times H^{1,0})\cap L^4((t,\infty),W^{1,\infty}\times W^{1,\infty})}\\
&\leq C\norm{\mathcal{K}(\eta_1+\mathcal{W})-\mathcal{K}(\eta_2+\mathcal{W})}_{L^1((t,\infty),H^{1,0}\times H^{1,0})}\\
&\leq C\norm{\mathcal{K}(\eta_1+\mathcal{W})-\mathcal{K}(\eta_2+\mathcal{W})}_{L^1((t,\infty),L^2)}+C\norm{\partial_x\mathcal{K}(\eta_1+\mathcal{W})-\partial_x\mathcal{K}(\eta_2+\mathcal{W})}_{L^1((t,\infty),L^2)}\\
&\leq C\int_t^{\infty} (\norm{\eta_1-\eta_2}_{L^2}+\norm{\partial_x\eta_1-\partial_x\eta_2}_{L^2})(\norm{(\eta_1,\eta_2,\mathcal{W})}_{L^{\infty}}^2+\norm{(\eta_1,\eta_2,\mathcal{W})}_{L^{\infty}}^p)\\
&\quad +C\norm{\eta_1-\eta_2}_{L^{\infty}}\norm{(\partial_x\eta_1,\partial_x\eta_2,\partial_x\mathcal{W})}_{L^2}(\norm{(\eta_1,\eta_2,\mathcal{W})}_{L^{\infty}}+\norm{(\eta_1,\eta_2,\mathcal{W})}_{L^{\infty}}^{p-1})d\tau\\
&\leq C(\norm{(\eta_1,\eta_2)}_{L^4((t,\infty),L^{\infty})}^2+\norm{\mathcal{W}}_{L^4((t,\infty),L^{\infty})}^2)\norm{\eta_1-\eta_2}_{L^2((t,\infty),H^1)}\\
&\quad +C(\norm{(\eta_1,\eta_2)}_{L^4((t,\infty),L^{\infty})}^3+\norm{\mathcal{W}}_{L^4((t,\infty),L^{\infty})}^3)(\norm{(\eta_1,\eta_2)}_{L^{\infty}((t,\infty),L^{\infty})}^{p-4}+\norm{\mathcal{W}}_{L^{\infty}((t,\infty),L^{\infty})}^{p-4})\\
&\quad\norm{\eta_1-\eta_2}_{L^4((t,\infty),L^{\infty})}\norm{(\partial_x\eta_1,\partial_x\eta_2,\partial_x\mathcal{W})}_{L^{\infty}((t,\infty),L^2)}\\
&\leq Cd(\eta_1,\eta_2)\left\{M^2t^{-2\alpha}+t^{-\frac{1}{2}}\left[\norm{\phi_+}_{2,1}+t^{-1}\ln(t) M_2(\phi_+)\right]^2\right\}Mt^{\frac{1}{2}-\alpha}\\
&\quad +Cd(\eta_1,\eta_2)\left\{M^3 t^{-3\alpha}+t^{-\frac{3}{4}}\left[\norm{\phi_+}_{2,1}+t^{-1}\ln(t) M_2(\phi_+)\right]^3\right\}(M^{p-4}t^{-\alpha(p-4)}+t^{-\frac{p-4}{2}})Mt^{-\alpha}(Mt^{-\alpha}+C_0),
\end{align*}
where the last inequality follows from \cite[the proof of Theorem 1]{HaOz94} and the constant $C_0$ depends on $\norm{\phi_+}_{4,0}+\norm{\phi_+}_{0,4}$. \\
Moreover, by letting $\eta_2=0$ in the above inequalities, we obtain
\begin{align*}
&\norm{\Phi(\eta_1)}_{L^{\infty}((t,\infty),H^{1,0})\cap L^4((t,\infty),W^{1,\infty})}\\
&\leq C(\norm{\eta_1}_{L^4((t,\infty),L^{\infty})}^2+\norm{\mathcal{W}}_{L^4((t,\infty),L^{\infty})}^2)\norm{\eta_1}_{L^2((t,\infty),H^1)}\\
&\quad +C(\norm{\eta_1}_{L^4((t,\infty),L^{\infty})}^3+\norm{\mathcal{W}}_{L^4((t,\infty),L^{\infty})}^3)(\norm{\eta_1}_{L^{\infty}((t,\infty),L^{\infty})}^{p-4}+\norm{\mathcal{W}}_{L^{\infty}((t,\infty),L^{\infty})}^{p-4})\\
&\quad\norm{\eta_1}_{L^4((t,\infty),L^{\infty})}\norm{(\partial_x\eta_1,\partial_x\mathcal{W})}_{L^{\infty}((t,\infty),L^2)}\\
&\leq C\left\{M^2t^{-2\alpha}+t^{-\frac{1}{2}}\left[\norm{\phi_+}_{2,1}+t^{-1}\ln(t) M_2(\phi_+)\right]^2\right\}Mt^{\frac{1}{2}-\alpha}\\
&\quad +C\left\{M^3 t^{-3\alpha}+t^{-\frac{3}{4}}\left[\norm{\phi_+}_{2,1}+t^{-1}\ln(t) M_2(\phi_+)\right]^3\right\}(M^{p-4}t^{-\alpha(p-4)}+t^{-\frac{p-4}{2}})Mt^{-\alpha}(Mt^{-\alpha}+C_0)\\
&\leq Ct^{-\alpha}(M^2t^{\frac{1}{2}-2\alpha}+\norm{\phi_+}^2_{2,1}+t^{-2}(\ln t)^2(M_2(\phi_+))^2+t^{-\frac{3}{4}}).
\end{align*}
Therefore, for $t$ sufficiently large and $\norm{\phi_+}_{2,1}$ sufficiently small, $\Phi$ is a contraction map on $X_{T,M}$. It follows that there exists a unique solution $\eta$ of \eqref{eq:eta} such that
\[
\sup_{t\geq T}t^{\alpha}\left(\norm{\eta}_{L^{\infty}((t,\infty),H^{1,0}\times H^{1,0})}+\norm{\eta}_{L^4((t,\infty),W^{1,\infty}\times W^{1,\infty})}\right) \lesssim 1.
\]
The rest of the proof follows from \cite{HaOz94}. Thus, we obtain that there exists a solution $\psi$ such that 
\begin{align*}
    (\text{MW}_+) \quad &\left\| \psi(t) - \exp(iS^+(t)) U(t)\phi_+ \right\|_{2,0}  \\
    &+ \left( \int_t^\infty \left\|\psi(\tau) - \exp(iS^+(\tau)) U(\tau)\phi_+ \right\|_{W^{2,\infty}}^4 d\tau \right)^{1/4}  \\
    &= O(t^{-\alpha}) \quad \text{as} \quad t \to \infty, 
\end{align*}
Since $\mathcal{G}u=\psi$, the proof of $(MW_+)$ is complete. $(MW_{-})$ is proved similarly. This completes the proof of Theorem \ref{thm:modified scattering operator}. 
\end{proof}

%\begin{theorem} \label{thm:non exists non trivial scattering solutions}
%Let $p\geq 4$. We assume that $u(t, x)$ is a solution of \eqref{eq1} satisfying
%\begin{align*}
%\makebox[1.5cm][l]{(\text{A.1})} \hfill & u(t, x) \in C(\mathbb{R}; H^{1,0}) \hfill %\\
%\intertext{and}
%\makebox[1.5cm][l]{(\text{A.2})} \hfill & \|u(t)\|_{1,0} \le C(\|u(0)\|_{1,0}), \quad \|u(t)\|_{4} = O(|t|^{-\frac{1}{4}}) \quad \text{as} \quad t \to \infty. \hfill \\
%\intertext{Moreover, we assume that there exists $\phi_+ \in H^{4,0} \cap H^{0,4}$ such that}
%\makebox[1.5cm][l]{(\text{A.3})} \hfill & \|(\mathcal{G} u)(t) - U(t)\phi_+\|_{2} \to 0 \quad \text{as} \quad t \to \infty. \hfill
%\end{align*}
%Then $u(t, x)$ is identically zero.
%\end{theorem}

\begin{proof}[Proof of Theorem \ref{thm:non exists non trivial scattering solutions}]
Let $\psi=\mathcal{G}u$. From \cite[Theorem 1.2]{HaOz94}, $\psi$ is identically zero. This implies the desired result.
\end{proof}

\section{Appendix}\label{sec:appendix}
This section is devoted to introducing well-known results that are useful for proving the main results in the preceding sections.
\begin{lemma}(Taylor expansion)\label{Taylor expansion}
If $f$ is $n$-times continuously differentiable on $[a, b]$ and $(n+1)$-times differentiable on $(a, b)$, then:$$f(x) = f(x_0) + \frac{f'(x_0)}{1!}(x - x_0) + \dots + \frac{f^{(n)}(x_0)}{n!}(x - x_0)^n + \frac{f^{(n+1)}(c)}{(n+1)!}(x - x_0)^{n+1}$$(where $c$ lies between $x$ and $x_0$).
\end{lemma}

Let $A$ be an bounded linear operator on a Banach space. We denote the spectrum of the operator $A$ by $\sigma(A)$. The spectral radius of $A$ is the supremum of the magnitudes of the elements of the spectrum:
\[
\rho(A)=\sup_{\lambda\in\sigma(A)}|\lambda|.
\]
We have the following result, which is known as Gelfand's formula.
\begin{lemma}
\label{lm:Gelfand formula}
For bounded linear operator $A$, its spectral radius is given by
\[
\rho(A)=\lim_{n\rightarrow\infty}\norm{A^n}^{\frac{1}{n}}=\inf_{n\in\N^*}\norm{A^n}^{\frac{1}{n}}.
\]
\end{lemma}

When $A$ is a self-adjoint operator, we have a simple formula for its spectral radius.
\begin{lemma}\label{lm:Gel_for_adjoint}
Let $A$ be an bounded linear operator on a Hilbert space $X$. Assume that $A$ is self-adjoint. Then, its spectral radius equals to its norm i.e 
\[
\rho(A)=\norm{A}.
\]
\end{lemma}

\begin{proof}
We have 
\[
\norm{A^2x}\leq \norm{A}\norm{Ax}\leq\norm{A}^2\norm{x},
\]
for all $x\in X$. Thus, $\norm{A^2}\leq\norm{A}^2$. \\
Moreover,
\[
\norm{Ax}^2=\left<Ax,Ax\right>\leq \left<A^2x,x\right>\leq \norm{A^2x}\norm{x}\leq \norm{A^2}\norm{x}^2.
\]
This implies that $\norm{A}^2\leq\norm{A^2}$ and hence the desired result is proved. 
\end{proof}

\section*{Acknowledgement} The author is supported by the VIASM Postdoctoral Fellowship. I would like to express my gratitude to VIASM for their warm hospitality and for providing a wonderful environment during my stay.

\bibliographystyle{abbrv}
\bibliography{paper}

\end{document}